\documentclass[a4paper, 12pt]{article}

\usepackage{amsmath, amssymb, eucal, exscale, mathrsfs, bm, ulem, amsthm} 

\usepackage{color}
\usepackage{dsfont}
\usepackage{cancel}

\numberwithin{equation}{section}

\newcommand{\form}{\mathscr E}
\newcommand{\dom}{D(\mathscr E)}
\newcommand{\real}{{\mathbb R}}
\newcommand{\dis}{\displaystyle}

\newtheorem{theorem}{Theorem}[section]
\newtheorem{lem}[theorem]{Lemma}
\newtheorem{prop}[theorem]{Proposition}
\newtheorem{cor}[theorem]{Corollary}

\newtheorem{exam}{Example}[section]

\theoremstyle{definition}

\newtheorem{rem}[theorem]{Remark}

\allowdisplaybreaks[4]

\title{\sf Energy integrals and asymmetric co-potentials for closed forms}
\date{}
\author{
{\sc Kazuhiro Kuwae}\footnote{Kazuhiro Kuwae ({\sf kuwae@fukuoka-u.ac.jp})
Department of Applied Mathematics, Fukuoka University,
Fukuoka 814-0180, Japan. 
Supported in part by JSPS Grant-in-Aid for Scientific Research (S) (No. 25K24482).
}, 
\ \ \ 
{\sc Takumu Ooi}\footnote{Takumu Ooi ({\sf ooitaku@rs.tus.ac.jp}) 
Department of Mathematics, Faculty of Science and Technology, Tokyo University of Science, Noda, Chiba, 278-8510, Japan. 
Supported in part by JSPS Grant-in-Aid for Early-Career Scientists (No. 25K17270).
}, 
\ \ \ 
{\sc Kaneharu Tsuchida}\footnote{Kaneharu Tsuchida ({\sf tsuchida@nda.ac.jp}) 
Department of Mathematics, National Defense Academy, Yokosuka, Kanagawa, Japan.
},
\\
\ \ and\ \
{\sc Toshihiro Uemura}\footnote{Toshihiro Uemura ({\sf t-uemura@kansai-u.ac.jp})
Department of Mathematics, Faculty of Engineering Science, Kansai University, Suita, Osaka 564-8680, Japan. Supported in part by JSPS Grant-in-Aid for Scientific Research (C) (No. 25K07056). 
}
}

\begin{document}
\baselineskip 16pt

\maketitle

\begin{abstract}
We investigate the class of measures of finite energy integrals and the behavior of potentials and co-potentials associated with non-symmetric closed forms. In particular, we compare these objects with their symmetric counterparts from three viewpoints: a non-symmetric version of Stollmann--Voigt's inequality, non-symmetric perturbations of symmetric forms, and closed forms associated with non-symmetric jump-type forms. Our results indicate that measures of finite energy integrals, potentials, and co-potentials behave differently in the non-symmetric setting, requiring more delicate analysis than in the symmetric case.
\end{abstract}



\section{Introduction}
The theory of Dirichlet forms and associated Markov processes provides a framework for studying potential theory. A central concept in this framework is the class of Radon measures called measures of finite energy integrals. In the classical theory of symmetric Dirichlet forms developed by Fukushima, Oshima, and Takeda \cite{FOT11}, this class plays a fundamental role in connecting analysis and probability. Analytically, for a closed form, a measure of finite energy integrals is a measure charging no set of zero capacity, and it ensures the existence and uniqueness of a finite energy potential whose energy represents the integral of quasi-continuous functions with respect to the measure. Probabilistically, via the Revuz correspondence, such a measure induces a positive continuous additive functional (PCAF) of the associated symmetric Markov process. Thus, the potential serves as a bridge between capacity and probabilistic dynamics.

When extending the theory to non-symmetric settings, particularly lower bounded closed forms and semi-Dirichlet forms $(\mathscr{E}, D(\mathscr{E}))$ on an $L^2$-space, the potential theoretic structures become richer and more delicate. As developed by Oshima \cite{O13}, analytic potential theory and the associated stochastic calculus remain available even when the dual semigroups are only positivity preserving rather than Markovian. Under the weak sector condition, a single measure $\nu$ of finite energy integrals gives rise to two distinct objects: the potential $U_\beta \nu$ and the co-potential $\widehat{U}_\beta \nu$, corresponding to the forward and dual semigroups, respectively. They are characterized by the identity $\mathscr{E}_\beta(U_\beta \nu, v)=\mathscr{E}_\beta(v,\widehat{U}_\beta \nu)=\int_E v\,d\nu$ for test functions $v$, where $\mathscr{E}_\beta$ denotes the sum of $\mathscr{E}$ and $\beta$ times the underlying $L^2$-inner product. A property we note early in this paper is that the total energies of these two 
potentials coincide from a macroscopic perspective: $\mathscr{E}_\beta(U_\beta \nu, 
U_\beta \nu) = \mathscr{E}_\beta(\widehat{U}_\beta \nu, \widehat{U}_\beta \nu)$. 
However, due to the asymmetry of a form, their pointwise 
behaviors can differ. For instance, even if the potential remains 
bounded, the co-potential may diverge at singularities or exhibit different decay 
rates at spatial infinity. This discrepancy arises because the antisymmetric part 
of the form acts as a source of energy that breaks the symmetry, distorting the 
spatial distribution of the potentials. The fine structures of these localized 
integral terms reveal the exact pointwise differences between the right and left 
potentials.

In this paper, we investigate the class of measures of finite energy integrals and analyze the differences between potentials and co-potentials for a non-symmetric form. In particular, we compare their properties in the non-symmetric and symmetric settings from three perspectives: Stollmann--Voigt's inequality, a non-symmetric perturbation of a symmetric form, a closed form associated with a non-symmetric jump kernel. Through these analyses, we show that the class of measures of finite energy integrals, as well as potentials and co-potentials, exhibits properties different from those in the symmetric case, highlighting the need for more delicate arguments in the non-symmetric setting. We present our results from these three viewpoints as follows.

As a theoretical basis, we establish a non-symmetric version of Stollmann--Voigt's inequality. Originally introduced by Stollmann and Voigt \cite{SV96} through an operator-theoretical approach, this inequality was later improved for symmetric Dirichlet forms in \cite{F00, ST05, BA04} for example, and recently proved by an alternative probabilistic method in \cite{OTU26-2}. We extend it to the non-symmetric setting of semi-Dirichlet forms. Analytically, this inequality yields a compact embedding from the Hilbert space induced by a form into an $L^2$-space with respect to the measure. In the symmetric case, the operator norm of the inclusion map is given by the $L^\infty$-norm of the associated potential, whereas in our non-symmetric setting, both the potential and co-potential determine this bound.

We investigate the differences in the structure of measures of finite energy integrals between symmetric and non-symmetric cases. A typical way to construct a non-symmetric closed form $\mathscr{E}$ is by adding a non-symmetric perturbation $b$ to a symmetric closed form $\mathscr{E}^0$. We examine the stability of the class $\mathsf{S}_0$ of Radon measures of finite energy integrals under such perturbations and establish analytical conditions under which the class remains invariant, $\mathsf{S}_0(\mathscr{E}^0)=\mathsf{S}_0(\mathscr{E})$, or becomes strictly larger, $\mathsf{S}_0(\mathscr{E}^0)\subsetneq\mathsf{S}_0(\mathscr{E})$. We show how the relative strength of the non-symmetric drift and the symmetric diffusion affects the structure of the finite energy class.

As another typical example of a non-symmetric form, we consider forms associated with non-symmetric jump kernels. In particular, we show that taking the transpose of the kernel preserves the structure of the class of measurs of finite energy 
integrals, while resulting in different potentials and co-potentials.

The present paper is organized as follows.
In Section \ref{sec_preparation}, we recall the framework of lower bounded closed forms and define the class $\mathsf{S}_0$ of measures of finite energy integrals. We also discuss the coincidence of the energies of potentials and co-potentials, and provide an example of a diffusion process with a singular drift illustrating spatial asymmetry.
Section \ref{sec_SVineq} is devoted to establishing a non-symmetric version of Stollmann--Voigt's inequality for semi-Dirichlet forms.
In Section \ref{sec_smallpertub}, we investigate the stability of the class $\mathsf{S}_0$ under non-symmetric perturbations. We show that, under certain conditions, the class of measures of finite energy integrals and the topology induced by it remain invariant, and provide examples.
In Section \ref{sec_largepurturb}, we present a criterion and a model showing that non-symmetric perturbations can enlarge the class of measures of finite energy integrals.
In Section \ref{sec_jump}, we consider non-symmetric jump-type forms and forms associated with their transposed jump kernels. We show that the pointwise differences between potentials and co-potentials are governed by the anti-symmetric part of the jump kernel, and provide an example of a stable-like process exhibiting distinct spatial decay rates at infinity.

\section{Measures of finite energy integrals for non-symmetric closed forms}\label{sec_preparation}

In this section we recall a lower bounded closed form  and semi-Dirichlet  form following \cite{O13}. 
To this end, let $(E,{\sf d})$ be a locally compact separable metric space and ${
  {\mathfrak{m}}}$ a positive Radon measure on $E$ with full topological support.
Let $\dom$ be a dense linear subspace of $L^2(E):=L^2(E;\mathfrak{m})$.  Denote by $(u,v)$ (resp. $\|u\|_{L^2}$) the inner product of $u$ and $v$ 
(resp. the $L^2$-norm of $u$).  A bilinear form $\form$ defined on $\dom\times \dom$ is called {\it a lower bounded closed form} on 
$L^2(E)$ if the following conditions $(\form.1), (\form.2), (\form.3)$ are satisfied: there exists $\beta_0\ge 0$ such that
\begin{itemize}
\item[$(\form.1)$] {\sf (lower boundedness):} For any $u \in \dom$, $\form_{\beta_0}(u,u)\ge 0$, where 
$$
\form_\beta(u,v):=\form(u,v) + \beta (u,v) \quad u,v \in \dom, \ \beta\ge 0.
$$
\item[$(\form.2)$] {\sf (strong sector condition):}  There exists a constant $K\ge 1$ such that 
$$
\big|\form(u,v)\big| \le K \sqrt{\form_{\beta_0}(u,u)} \sqrt{\form_{\beta_0}(v,v)} \quad u,v \in \dom.
$$
\item[$(\form.3)$] {\sf (closedness):} The space $\dom$ is a real Hilbert space relative to the inner product
$$
\form^{(s)}_\beta(u,v)=\frac 12 \Big(\form_{\beta}(u,v) +\form_{\beta}(v,u)\Big) \quad {\rm for \ all} \ \beta >\beta_0.
$$
\end{itemize}
 By utilizing $(\form.1)$,  we see that for any $u \in \dom$ and $\beta > \beta_0$, the relation
$$
\form_\beta(u,u) = \form_{\beta_0}(u,u) + (\beta-\beta_0)\|u\|_{L^2}^2 \geq 0
$$
holds. This immediately implies the following form equivalence for any $\alpha, \beta > \beta_0$:
\begin{equation} \label{form-top}
\Big( 1 \wedge \frac{\beta-\beta_0}{\alpha-\beta_0}\Big) \form_\alpha(u,u) \le \form_\beta(u,u) \le \Big( 1 \vee \frac{\beta-\beta_0}{\alpha-\beta_0}\Big) \form_\alpha(u,u), \quad u \in \dom.
\end{equation}
This means that the topology of the Hilbert space $\dom$ endowed with the inner product $\form^{(s)}_\beta(u,v)$ does not depend on the choice of $\beta > \beta_0$.   When $(\form.1)$ holds for $\beta_0=0$, the pair $(\form, \dom)$ is simply called a {\it nonnegative closed form} on $L^2(E)$.

The condition $(\form.2)$ yields the following condition $(\form.2)'$ 
with $K_{\beta}=K+\beta/(\beta-\beta_0)$: 
\begin{itemize}
\item[$(\form.2)'$] {\sf (weak sector condition):} For each $\beta>\beta_0$, there 
exists $K_{\beta}\geq1$ such that 
$$
\big|\form_{\beta}(u,v)\big| \le K_{\beta} \sqrt{\form_{\beta}(u,u)} \sqrt{\form_{\beta}(v,v)} \quad u,v \in \dom.
$$
\end{itemize}

For a lower bounded closed form $(\form, \dom)$ on $L^2(E)$ with a parameter $\beta_0\ge 0$,  there exist unique semigroups
$\{T_t; t>0\}$ and $\{\widehat{T}_t; t>0\}$ of linear operators on $L^2(E)$ satisfying 
$$
(T_t f,g)=(f, \widehat{T}_t g), \ \  \|T_t f\|_{L^2} \le e^{\beta_0 t}\|f\|_{L^2}, \ \ \|\widehat{T}_t f\|_{L^2} \le e^{\beta_0 t}\|f\|_{L^2} 
\ \ {\rm for} \ f,g \in L^2(E), \ t>0
$$
such that their Laplace transforms $\{G_\beta\}$ and $\{\widehat{G}_{\beta}\}$ are determined  for $\beta >\beta_0$ by 
$$
G_\beta f, \widehat{G}_\beta f\in \dom, \ \ 
\form_\beta(G_\beta f, u)= \form_{\beta}(u, \widehat{G}_\beta f)=(f,u) \ \ {\rm for} \  f\in L^2(E), \ u \in \dom.
$$
The semigroup $\{T_t; t>0\}$ is said to be {\it Markovian} if $0\le T_t f \le 1, \ t>0$ whenever $f\in L^2(E)$ with $0\le f\le 1$.
Then it is known that $\{T_t;t>0\}$ is Markovian if and only if the following condition is satisfied:

\begin{itemize}
\item[$(\form.4)$] {\sf (Markov property):} For all $u\in D(\mathscr{E})$ and $a\geq0$, $u\land a\in D(\mathscr{E})$ and $\mathscr{E}(u\land a,u-u\land a)\geq0$. 
\end{itemize}
A lower bounded closed form $(\form, \dom)$ on $L^2(E)$ satisfying $(\form.4)$ is called  {\it a lower bounded semi-Dirichlet form} on $L^2(E)$. 

A lower bounded closed form (resp.  lower bounded semi-Dirichlet form) on $L^2(E)$ associated with a parameter $\beta_0 \ge 0$ is simply called a {\it closed form} (resp. {\it semi-Dirichlet form}) {\it with the lower bound  $-\beta_0$}. In the case where $\beta_0 = 0$, we omit the parameter and refer to it as a {\it nonnegative closed form} (resp. {\it nonnegative semi-Dirichlet form}).

\smallskip
We say that a  closed form $(\form, \dom)$ on $L^2(E)$ with a lower bound  $-\beta_0$ admits a {\it core} $\mathscr{C}$ if $\mathscr{C}$ is a linear subspace of $\dom \cap C_0(E)$ such that $\mathscr{C}$ is dense both in $\dom$ with respect to the $\form_\beta$-norm for $\beta > \beta_0$ and in $C_0(E)$ with respect to the uniform norm $\|\cdot\|_\infty$. Here, $C_0(E)$ denotes the space of all continuous functions on $E$ with compact support. 

It is worth noting that while the existence of a core is defined here for general closed forms, if $(\form, \dom)$ happens to be a (semi-)Dirichlet form, this requirement exactly coincides with the standard definition of being {\it regular}.

Now we briefly recall the notion of capacity and quasi-continuity associated with the semi-Dirichlet form $(\mathscr{E}, D(\mathscr{E}))$. 
Following \cite{O13}, for any open set $A \subset E$, we set 
$\mathscr{L}_{A} := \{u \in D(\mathscr{E}) \mid u \ge 1\; \mathfrak{m}\text{-a.e. on } A\}$. 
If $\mathscr{L}_{A} \neq \emptyset$, there exist a unique $\alpha$-equilibrium potential $e^\alpha_A \in \mathscr{L}_{A}$ 
and an $\alpha$-coequilibrium potential $\widehat{e}^\alpha_A \in \mathscr{L}_{A}$. 
We define the $\alpha$-capacity of $A$ by 
$\text{Cap}^{(\alpha)}(A) := \mathscr{E}_\alpha(e^\alpha_A, \widehat{e}^\alpha_A)$ 
(and $\infty$ if $\mathscr{L}_{A} = \emptyset$). 
For an arbitrary subset $B \subset E$, its capacity is defined by 
$\text{Cap}^{(\alpha)}(B) := \inf \{ \text{Cap}^{(\alpha)}(A) \mid A \text{ is open, } A \supset B\}$, 
which is known to be a Choquet capacity.

A statement is said to hold quasi-everywhere (q.e.) if it holds outside a set of capacity zero. 
A function $u$ on $E$ is called quasi-continuous if for any $\varepsilon > 0$, 
there exists an open set $G \subset E$ with $\text{Cap}^{(\alpha)}(G) < \varepsilon$ 
such that the restriction of $u$ to $E \setminus G$ is continuous. 
Under our regularity assumption, every $u \in D(\mathscr{E})$ admits a quasi-continuous $\mathfrak{m}$-version.  
Hereafter, we always assume that every element $u \in \dom$ is replaced by its quasi-continuous $\mathfrak{m}$-version.

\bigskip

Assuming that the closed form $(\form, \dom)$ on $L^2(E)$ with the lower bound  $-\beta_0$ admits a core $\mathscr{C}$, we next introduce the class of measures of finite energy integrals associated 
with this framework. We remark that we do not assume the Markov property $(\form.4)$. Let $\mathscr{B}(E)$ be the $\sigma$-algebra of all Borel sets on $E$.  
A positive Radon measure $\nu$ on $E$ is said to be a measure of  {\it finite energy integral} (or to belong to the class $\mathsf{S}_0 = \mathsf{S}_0(\form)$) 
if for any $\beta > \beta_0$, there exist a constant $C > 0$  such that
\begin{equation}\label{eq:S0-def}
\int_E |v(x)| \, \nu(dx) \le C \sqrt{\form_\beta(v,v)}=C \sqrt{\form^{(s)}_\beta(v,v)} \quad \text{for all } v \in \dom \cap C_0(E).
\end{equation}

According to the Lax--Milgram theorem, for each $\nu \in \mathsf{S}_0$ and $\beta > \beta_0$, there exists a unique element $U_\beta \nu \in \dom$ (called the {\it potential} of $\nu$) and a unique element $\widehat{U}_\beta \nu \in \dom$ (called the {\it co-potential} of $\nu$) satisfying the following relations for all $v \in \dom \cap C_0(E)$:
\begin{equation} \label{S_0}
\form_\beta(U_\beta \nu, v) = \form_\beta(v, \widehat{U}_\beta \nu) = \int_E v(x) \, \nu(dx).
\end{equation}

  Note that while the integral $\int_{E}^{ }v d\nu$ depends on the pointwise values of $v$
  (which is why we assume a quasi-continuous version for a general $v \in \dom$),
  the inner product relation $\form_\beta(U_\beta \nu, v) = \form_\beta(v, \widehat{U}_\beta \nu)$
  is free from such pointwise modifications.

Therefore, by the denseness of the core $\mathscr{C}$ in $\dom$ with respect to the $\form_\beta$-norm, the first equality in \eqref{S_0} successfully extends to the entire domain without any ambiguity:
\begin{equation} \label{S_0-0}
\form_\beta(U_\beta \nu, v) = \form_\beta(v, \widehat{U}_\beta \nu) \quad \text{for all } v \in \dom.
\end{equation}
This extension establishes the foundational algebraic identities and topological properties of $\mathsf{S}_0$, summarized in the following proposition.

\begin{prop}\label{prop:S_0-properties}
For any $\mu, \nu \in \mathsf{S}_0$ and $\beta > \beta_0$, the total accumulated energies coincide:
\begin{equation}\label{eq:macro-energy-coincidence}
\form_\beta(U_\beta \nu, U_\beta \nu) = \form_\beta(\widehat{U}_\beta \nu, \widehat{U}_\beta \nu) = \form_\beta(U_\beta \nu, \widehat{U}_\beta \nu),
\end{equation}
and the energy of the potential differences satisfies
\begin{equation}\label{eq:energy-difference-coincidence}
\form_\beta(U_\beta \mu - U_\beta \nu, U_\beta \mu - U_\beta \nu) = \form_\beta(\widehat{U}_\beta \mu - \widehat{U}_\beta \nu, \widehat{U}_\beta \mu - \widehat{U}_\beta \nu).
\end{equation}
Furthermore, the function $\rho_\beta \colon \mathsf{S}_0 \times \mathsf{S}_0 \to [0, \infty)$ defined by
\begin{equation}\label{eq:metric-def}
\rho_\beta(\mu, \nu) := \sqrt{\form_\beta (U_\beta \mu - U_\beta \nu, U_\beta \mu - U_\beta \nu)} = \sqrt{\form_\beta (\widehat{U}_\beta \mu - \widehat{U}_\beta \nu, \widehat{U}_\beta \mu - \widehat{U}_\beta \nu)}
\end{equation}
constitutes a metric on $\mathsf{S}_0$. This metric satisfies the following equivalence relation for any $\alpha, \beta > \beta_0:$
  \begin{equation}\label{eq:metric-equivalence}
\bigg(K_\beta \sqrt{1 \vee \frac{\beta-\beta_0}{\alpha-\beta_0}} \bigg)^{-1} \rho_\alpha(\mu, \nu) \le \rho_\beta(\mu, \nu)
\le  \bigg(K_\alpha\sqrt{1 \vee \frac{\alpha-\beta_0}{\beta-\beta_0}} \bigg) \rho_\alpha(\mu, \nu). 
\end{equation}
In particular, the topology induced by $\rho_\beta$ is independent of the choice of $\beta > \beta_0$, and $(\mathsf{S}_0, \rho_\beta)$ forms a Polish space.
\end{prop}
\begin{proof}
The identities \eqref{eq:macro-energy-coincidence} and \eqref{eq:energy-difference-coincidence} follow immediately from the extended relation \eqref{S_0-0} by sequentially substituting $v = U_\beta \nu$, $v = \widehat{U}_\beta \nu$, and their linear combinations. 

For \eqref{eq:metric-equivalence} it suffices to show the upper bound, as the lower bound follows by exchanging the roles of $\alpha$ and $\beta$.
Fix $\mu, \nu \in \mathsf{S}_0$. Let $v := U_\beta \mu- U_\beta \nu \in \dom$. Since the core $\mathscr{C}$ is dense in $\dom$ with respect to the energy norm, there exists a sequence $\{v_n\} \subset \mathscr{C}$ such that $v_n \to v$ strongly in $(\dom, \form_\beta)$. Due to the norm equivalence (2.1), $v_n$ also converges to $v$ strongly in $(\dom, \form_\alpha)$.
By the continuity of the bilinear form and the definition of the potentials (2.3) for the test functions $v_n \in \mathscr{C}$, we have
\begin{align*}
\rho_\beta(\mu, \nu)^2 & = \form_\beta(U_\beta \mu- U_\beta \nu, v) = \lim_{n \to \infty} \form_\beta(U_\beta \mu- U_\beta \nu , v_n)  \\
&=  \lim_{n \to \infty} \Big(\int_E v_n d\mu- \int_E v_n d\nu \Big)= \lim_{n \to \infty} \form_\alpha(U_\alpha \mu-U_\alpha\nu, v_n).
\end{align*}
Applying the weak sector condition $(\form.2)'$ for $\form_\alpha$, we can estimate the right-hand side as
\begin{align*}
\big|\form_\alpha(U_\alpha \mu-U_\alpha \nu, v_n)\big| & \le K_\alpha \form_\alpha(U_\alpha \mu-U_\alpha \nu, U_\alpha \mu - U_{\alpha}\nu)^{1/2} \form_\alpha(v_n, v_n)^{1/2} \\
& = K_\alpha \rho_\alpha(\mu, \nu) \form_\alpha(v_n, v_n)^{1/2}.
\end{align*}
Using the form equivalence (2.1), we can bound the $\form_\alpha$-norm of $v_n$ by its $\form_\beta$-norm:
\begin{equation*}
\form_\alpha(v_n, v_n) \le \left(1 \vee \frac{\alpha-\beta_0}{\beta-\beta_0}\right) \form_\beta(v_n, v_n).
\end{equation*}
Taking the limit as $n \to \infty$, the strong convergence $v_n \to v$ in $\form_\beta$ yields $\form_\beta(v_n, v_n)^{1/2} \to \form_\beta(v, v)^{1/2} = \rho_\beta(\mu, \nu)$. Thus, we obtain
\begin{equation*}
\rho_\beta(\mu, \nu)^2 \le K_\alpha \sqrt{1 \vee \frac{\alpha-\beta_0}{\beta-\beta_0}} \rho_\alpha(\mu, \nu) \rho_\beta(\mu, \nu).
\end{equation*}
Dividing both sides by $\rho_\beta(\mu, \nu)$ yields the desired upper bound. The lower bound is obtained by entirely the same argument.
The other metric properties follow similarly to [NTTU25, Proposition 3.6] by taking \eqref{eq:energy-difference-coincidence} into account.
\end{proof}

\begin{rem}\rm
In the case of non-symmetric semi-Dirichlet forms, the convergence of the corresponding PCAFs when measures of finite energy integrals converge with respect to this distance \(\rho_\beta\) are investigated in our forthcoming paper \cite{KOTU}.
\end{rem}

According to the previous proposition,  we can show  the following  properties of the class $\mathsf{S}_0$ concerning vague and weak convergences  
similarly to \cite[Propositions 3.8 and 3.9]{NTTU25}. 

\begin{prop} \rm 
Take $\beta>\beta_0$ and fix it. Then the following properties hold for the class $\mathsf{S}_0$ and the metric $\rho_\beta$:  
\begin{enumerate}
\item[(i)] \textbf{(Vague convergence)} 

If $\rho_\beta(\mu_n, \mu) \to 0$ for $\mu_n, \mu \in \mathsf{S}_0$, then $\mu_n$ converges to $\mu$ vaguely.
\item[(ii)] \textbf{(Vague convergence with bounded energy implies weak convergence)}

If a sequence $\{\mu_n\} \subset \mathsf{S}_0$ is bounded in \(\rho_\beta\), namely $\sup_{n} \form_\beta(U_\beta \mu_n, U_\beta \mu_n) < \infty$, and converges vaguely to a Radon measure $\mu$, then $\mu \in \mathsf{S}_0$ and the potential $U_\beta \mu_n$ (resp.\ the co-potential $\widehat{U}_\beta \mu_n$) converges to $U_\beta \mu$ (resp.\ $\widehat{U}_\beta \mu$) weakly in the Hilbert space $(\dom, \form^{(s)}_\beta)$.
\end{enumerate}
\end{prop}

\begin{proof}
Since the proof of (i) is quite similar to that of \cite[Proposition 3.8]{NTTU25}, we omit it. 

For the second statement, the proof requires a slight modification because of the non-symmetry. Assume that $\sup_n \form_\beta(U_\beta \mu_n, U_\beta \mu_n) < \infty$ and $\mu_n \to \mu$ vaguely. Since the diagonal components coincide with $\form^{(s)}_\beta(u, u) = \form_\beta(u, u)$ for any $u \in \dom$, the sequence $\{U_\beta \mu_n\}$ is bounded in the Hilbert space $(\dom, \form^{(s)}_\beta)$. By the Banach-Alaoglu theorem, there exist a subsequence $\{\mu_{n_k}\}$ and a function $u \in \dom$ such that $U_\beta \mu_{n_k}$ converges weakly to $u$ in $(\dom, \form^{(s)}_\beta)$. Due to the  weak sector condition, the maps $w \mapsto \form_\beta(w, v)$ and $w \mapsto \form_\beta(v, w)$ are continuous linear functionals on $(\dom, \form^{(s)}_\beta)$ for each fixed $v \in \dom$. Hence, the weak convergence implies $\form_\beta(U_\beta \mu_{n_k}, v) \to \form_\beta(u, v)$ for all $v \in \dom$. In particular, for any test function $v \in \dom \cap C_0(E)$, the vague convergence yields
\begin{equation*}
\form_\beta(u, v) = \lim_{k \to \infty} \form_\beta(U_\beta \mu_{n_k}, v) = \lim_{k \to \infty} \int_E v d\mu_{n_k} = \int_E v d\mu.
\end{equation*}
This identity ensures that $\mu \in \mathsf{S}_0$ and $u = U_\beta \mu$. Since the limit $U_\beta \mu$ is uniquely determined independently of the choice of the subsequence, the entire sequence $U_\beta \mu_n$ converges weakly to $U_\beta \mu$ in $(\dom, \form^{(s)}_\beta)$. By the exact same argument utilizing the total energy coincidence $\form_\beta(\widehat{U}_\beta \mu_n, \widehat{U}_\beta \mu_n) = \form_\beta(U_\beta \mu_n, U_\beta \mu_n)$ and the relation $\form_\beta(v, \widehat{U}_\beta \mu_n) = \int_E v d\mu_n$, the weak convergence of the co-potentials $\widehat{U}_\beta \mu_n \to \widehat{U}_\beta \mu$ is obtained equivalently.
\end{proof}

  Next, we introduce the notion of smooth measures. A positive Borel measure $\nu$ on $E$ is said to be a \textit{smooth measure} 
if it charges no set of zero capacity and there exists a generalized nest $\{F_n\}$ 
of closed sets such that $\nu(F_n) < \infty$ for all $n \in \mathbb{N}$. 
We denote the family of all smooth measures by $\mathsf{S}$. 
By \cite{O13}, any measure in $\mathsf{S}_0$ charges no set of zero capacity, 
which implies $\mathsf{S}_0 \subset \mathsf{S}$.

Under the regularity of the semi-Dirichlet form $(\mathscr{E}, D(\mathscr{E}))$, 
there exists a Hunt process $\mathbb{M} = (\Omega, (X_t)_{t\geq0}, (\mathbb{P}_x)_{x\in E})$ associated with it. 
For any smooth measure $\nu \in \mathsf{S}$,
there exists a unique positive continuous additive functional
(PCAF in short, see \cite[\S 4.1]{O13} for the definition)
$\mathsf{A}^\nu$ of $\mathbb{M}$ under the Revuz correspondence:
$$
\lim_{t \downarrow 0} \frac{1}{t} \int_E \mathbb{E}_x \left[ \int_0^t f(X_s) d{\mathsf A}^\nu_s \right] \widehat{h}(x) \mathfrak{m}(dx) 
= \int_E f(x) \widehat{h}(x) \nu(dx)
$$
for any $\gamma$-coexcessive function $\widehat{h} \in L^1(E)$ ($\gamma > 0$) 
and 
  $f\in \mathscr{B}_b(E)$ with $f\ge 0$. 
Furthermore, for $\nu \in \mathsf{S}_0$ and $\alpha > \beta_0$, the probabilistic potential defined by
$$
R_\alpha \nu(x) := \mathbb{E}_x \left[ \int_0^\infty e^{-\alpha t} d{\mathsf A}^\nu_t \right]
$$
gives a quasi-continuous $m$-version of the analytic potential $U_\alpha \nu$.

If the dual form $(\widehat{\mathscr{E}}, D(\mathscr{E}))$ also satisfies the Markov property, 
there exists a dual Hunt process $\widehat{\mathbb{M}} = (\widehat{\Omega}, (\widehat{X}_t)_{t\geq0}, (\widehat{\mathbb{P}}_x)_{x\in E})$ associated with it. 
In this case, for any $\nu \in \mathsf{S}$, we can similarly define the dual PCAF $\widehat{\mathsf{A}}^\nu$ of $\widehat{\mathbb{M}}$
under the dual Revuz correspondence. 
Its probabilistic co-potential $\widehat{R}_\alpha \nu(x) := \widehat{\mathbb{E}}_x \left[ \int_0^\infty e^{-\alpha t} d\widehat{\mathsf{A}}^\nu_t \right]$ 
gives a quasi-continuous $m$-version of the analytic co-potential $\widehat{U}_\alpha \nu$.

It is worth emphasizing that while the condition \eqref{eq:S0-def} for a measure belonging to $\mathsf{S}_0$ depends on the symmetric part $\form^{(s)}_\beta$, the actual point-wise structures of the potential $U_\beta \nu$ and the co-potential $\widehat{U}_\beta \nu$ are heavily asymmetric due to the underlying non-symmetry. Nevertheless, as we have just established in \eqref{eq:macro-energy-coincidence}, a cancellation occurs at the global energy level, keeping the total magnitudes identical.

To see where the microscopic discrepancy begins to manifest, it is highly instructive to look at the formal integral representations. Although evaluating the strict pointwise values under highly singular measures generally poses a severe problem regarding the choice of quasi-continuous versions, one can formally rewrite the total energy coincidence in the following mixed-integral form:
\begin{equation}\label{eq:energy-coincidence-pure}
\form_\beta(U_\beta \nu, U_\beta \nu) = \form_\beta(\widehat{U}_\beta \nu, \widehat{U}_\beta \nu) = \int_E \widehat{U}_\beta \nu(x) \, \nu(dx) + \int_E U_\beta \nu(x) \, \nu(dx) - \form_\beta(U_\beta \nu, \widehat{U}_\beta \nu).
\end{equation}
Heuristically, it is precisely within the fine structures of these localized integral terms that the pointwise discrepancy between the right and left potentials becomes explicit. This structural tension naturally motivates the concrete examples and applications presented below.

\subsection{Example for Spatial Asymmetry of Bounded Potentials and Co-potentials}

\smallskip
Let $D = \{|x| < R\}$ be the open ball of radius $R > 0$ centered at the origin in $\mathbb{R}^d$ with $d \ge 3$. 
Let $B(x)$ be a vector potential defined as
\begin{equation}\label{eq:ex-drift}
B(x) := \frac{cx}{|x|^2} \quad \text{for } x \in D \setminus \{0\} \quad \text{with}  \ \  c>0.
\end{equation}
Define a bilinear form $\form(u,v)$ by 
\begin{equation} \label{ex:form1}
\form(u,v):= \int_D \nabla u(x) \cdot \nabla v(x)dx + \int_D B(x) \cdot \nabla u(x) v(x)  \, dx, \quad 
u, v \in C_0^\infty(D),
\end{equation}
where \( C_0^\infty(D)\) is the space of all infinitely differentiable function with compact support. We first show this bilinear form $\form$  is not just well-defined for $C_0^\infty(D)$ but  also produces a nonnegative semi-Dirichlet 
form having $H^1_0(D)$ as its domain on $L^2(D)$: 
\begin{lem}
For $d \ge 3$ and $0 < c < \frac{d-2}2$, the bilinear form $\form$ defined in \eqref{ex:form1} is nonnegative semi-Dirichlet form with $H^1_0(D)$ as its domain on $L^2(D)$.
\end{lem}
\begin{proof}
Take $u \in C_0^\infty(D)$. Noting that ${\rm div} B(x)=c(d-2)|x|^{-2}$, we see that 
\begin{align*}
\int_D B(x) \cdot \nabla u(x) \, u(x)dx  & =\frac 12\int_D B(x) \cdot \nabla (u^2(x))dx   =- \frac 12 \int_D\Big( {\rm div} B(x)  \Big)u(x)^2 dx  \\
& =-\frac{c(d-2)}2 \int_D \frac{u(x)^2}{|x|^2}dx.
\end{align*}
By making use of Hardy's inequality, $\frac{(d-2)^2}4 \int_D \frac{u(x)^2}{|x|^2}dx \le \int_D |\nabla u(x)|^2dx=:{\mathbb D}(u,u)$, it follows 
$$
\int_D B(x) \cdot \nabla u(x) \, u(x)dx \ge -\frac{2c}{d-2} \int_D |\nabla u(x)|^2dx=-\frac{2c}{d-2}   {\mathbb D}(u,u), 
$$
and then, putting $0<c<\frac{d-2}2$,  we see that 
\begin{equation} \label{hardy1}
\Big(1-\frac{2c}{d-2}\Big){\mathbb D}(u,u) \le \form(u,u) \le \Big(1+\frac{2c}{d-2}\Big) {\mathbb D}(u,u).
\end{equation}
This implies that the form $\form$ is not just nonnegative definite but is equivalent to the Dirichlet integral \(\mathbb{D}\).

Furthermore, the inequality established above immediately yields the closedness of the form $(\form, H^1_0(D))$. 
Let $\form_1(u,u) := \form(u,u) + \|u\|_{L^2}^2$. By combining \eqref{ex:form1} with the bound \eqref{hardy1}, it is straightforward to see that there exist positive constants $C_1, C_2 > 0$ such that
\begin{equation}\label{eq:norm-equivalence}
C_1 \mathbb{D}_1(u,u) \le \form_1(u,u) \le C_2 \mathbb{D}_1(u,u), \quad \forall u \in H^1_0(D),
\end{equation}
where $\mathbb{D}_1(u,u) := \|\nabla u\|_{L^2}^2 + \|u\|_{L^2}^2$ denotes the standard Dirichlet (Sobolev) norm. Since $H^1_0(D)$ is complete under this norm, the equivalence \eqref{eq:norm-equivalence} implies that $(\form, H^1_0(D))$ is a  nonnegative definite  closed form on $L^2(D)$. 

In addition to the closedness, the form $\form$ satisfies the Markov property $(\form.4)$.  Indeed, as ${\rm div} \, B =c(d-2)|x|^{-2} \ge  0$, controls the non-symmetric perturbation distributively, by the integration by parts formula, it follows that the contraction inequality
$$
\form(u^{\#}, u-u^{\#})\ge 0
$$
holds for all $u \in H^1_0(D)$, \(a\geq 0\), and $u^{\#} := u\wedge a$. Therefore, $(\form, H^1_0(D))$ is verified to be a nonnegative semi-Dirichlet form on $L^2(D)$. 
\end{proof}
We now consider a singular Borel measure 
$$
\nu(dx):=|x|^{-1}dx.
$$ 
We will prove that $\nu \in \mathsf{S}_0$ and its potential $U_1\nu$ is in $L^\infty(D)$, but its co-potential $\widehat{U}_1\nu$ is not in $L^\infty(D)$.

To verify that $\nu$ is a measure of finite energy integrals, we evaluate its action on any test function $v \in H_0^1(D)\cap C_0$. By the Cauchy--Schwarz inequality and Hardy's inequality, we have
\begin{align*}
\int_D |v(x)| \nu(dx) &= \int_D |v(x)| |x|^{-1} dx \\
&\le \left( \int_D \frac{v(x)^2}{|x|^2} dx \right)^{1/2} \left( \int_D 1 dx \right)^{1/2} \\
&\le \frac{2}{d-2} \|\nabla v\|_{L^2(D)} {\sf vol}(D)^{1/2}, 
\end{align*}
where ${\sf vol}(D)$ is the volume of $D$. Since $\form_1(v, v) \ge \form(v, v) \ge \left( 1 - \frac{2c}{d-2} \right) \|\nabla v\|_{L^2(D)}^2$, there exists a constant $C > 0$ independent of $v$ such that $\int_D |v| d\nu \le C \sqrt{\form_1(v, v)}$. Thus, $\nu \in \mathsf{S}_0$.

\smallskip
The 1-potential $u = U_1\nu \in H_0^1(D)$ satisfies the equation 
$$
-\Delta u + B \cdot \nabla u + u = |x|^{-1}
$$ 
in the weak sense. Due to the radial symmetry of the domain and the drift, the potential $u$ can be represented as a radial function $u(r)$ with $r = |x|$, which satisfies the ordinary differential equation:
\begin{equation}\label{eq:ode-radial}
u''(r) + \frac{d-1-c}{r} u'(r) - u(r) = -\frac{1}{r}, \quad 0 < r < R.
\end{equation}
To precisely analyze the asymptotic behavior and the singularity of the solution $u(r)$ near the origin, we rewrite \eqref{eq:ode-radial} in the Sturm-Liouville form:
\begin{equation}\label{eq:sturm-liouville}
\left( r^{d-1-c} u'(r) \right)' = r^{d-1-c} \left( u(r) - r^{-1} \right).
\end{equation}
By integrating \eqref{eq:sturm-liouville} from a small radius $r$ to a fixed $\varepsilon > 0$, we find that the general solution for the derivative $u'(r)$ near the origin must satisfy
\begin{equation}\label{eq:derivative-asymptotic}
u'(r) = C_0 r^{c-d+1} - \frac{1}{d-1-c} + o(1) \quad \text{as } r \to 0,
\end{equation}
where $C_0$ is an integration constant associated with the homogeneous part of the leading differential operator. 

We now test the admissibility of the leading singular term $C_0 r^{c-d+1}$ within the energy space. The condition $u \in H_0^1(D)$ requires the gradient to be square-integrable with respect to the speed measure near the origin, which translates to
\begin{equation}\label{eq:energy-integrability}
\int_0^\varepsilon |u'(r)|^2 r^{d-1} \, dr < \infty.
\end{equation}
Substituting the asymptotic expansion \eqref{eq:derivative-asymptotic} into \eqref{eq:energy-integrability}, the square-integrability of the singular component holds if and only if
$$
\int_0^\varepsilon \left( C_0 r^{c-d+1} \right)^2 r^{d-1} \, dr = C_0^2 \int_0^\varepsilon r^{2c-d+1} \, dr < \infty.
$$
This integral converges if and only if $2c - d + 1 > -1$, which simplifies directly to the parameter constraint $c > \frac{d-2}{2}$. However, under our fundamental assumption for the lower boundedness of the form, the parameter is strictly bounded by $0 < c < \frac{d-2}{2}$. This produces a complete contradiction unless the coefficients of the singular homogeneous solution vanish identically, enforcing $C_0 = 0$.

Consequently, the contribution of the singular homogeneous solution is eliminated. The asymptotic behavior of the derivative is uniquely governed by the particular solution component, yielding
$$
u'(r) = -\frac{1}{d-1-c} + o(1) \quad \text{as } r \to 0.
$$
Integrating this relation once more from $0$ to $r$ demonstrates that $u(r)$ is Lipschitz continuous in a neighborhood of the origin, which guarantees that $\lim_{r \to 0} u(r) < \infty$. Hence, we conclude that the $1$-potential stays uniformly bounded, establishing $U_1\nu \in L^\infty(D)$.

\smallskip
As for the co-potential $w = \widehat{U}_1\nu \in H_0^1(D)$, it satisfies the adjoint equation 
$$
-\Delta w - \text{div}(Bw) + w = |x|^{-1}.
$$ 
Since $\text{div}(Bw) = B \cdot \nabla w + w \, \text{div}B = \frac{c}{r} w' + \frac{c(d-2)}{r^2} w$, the radial profile $w(r)$ satisfies:
\begin{align*}
w''(r) + \frac{d-1+c}{r} w'(r) + \frac{c(d-2)}{r^2} w(r) - w(r) = -\frac{1}{r}.
\end{align*}
The principal part near $r=0$ is the Euler equation $w'' + \frac{d-1+c}{r} w' + \frac{c(d-2)}{r^2} w = 0$. By substituting $w(r) = r^\lambda$, 
we obtain the characteristic equation $\lambda^2 + (d-2+c)\lambda + c(d-2) = 0$, which factors as $(\lambda+c)(\lambda+d-2) = 0$. 
Thus, the fundamental solutions to the leading Euler operator are given by $r^{-c}$ and $r^{-(d-2)}$.

To analyze the true asymptotic behavior of the co-potential $w(r)$ near the origin, we multiply the full adjoint ODE by the integrating factor $r^{d-1+c}$ 
and rewrite it into the following Sturm-Liouville structure:
\begin{equation}
(r^{d-1+c} w'(r))' + c(d-2)r^{d-3+c} w(r) = r^{d-1+c} (w(r) - r^{-1}). \label{eq:w_sturm}
\end{equation}
Treating the right-hand side of \eqref{eq:w_sturm} as a lower-order perturbation and taking into account the fundamental solutions $r^{-c}$ and $r^{-(d-2)}$ of the homogeneous part, the general solution for the derivative $w'(r)$ in a neighborhood of the origin can be expressed via the asymptotic expansion:
\begin{equation}
w'(r) = C_1 r^{-c-1} + C_2 r^{-(d-2)-1} + O(1) \quad \text{as } r \to 0, \label{eq:w_deriv}
\end{equation}
where $C_1$ and $C_2$ are integration constants.

We now examine the square-integrability of each component under the energy norm constraint required for $w \in H_0^1(D)$, namely, $\int_0^\varepsilon |w'(r)|^2 r^{d-1} \, dr < \infty$. 
\begin{itemize}
    \item For the most singular component involving $C_2$, we see that
    $$
    \int_0^\varepsilon \left( r^{-(d-2)-1} \right)^2 r^{d-1} \, dr = \int_0^\varepsilon r^{-d+1} \, dr = \infty
    $$
    diverges {
      algebraically for $d \ge 3$ (logarithmically for $d=2$).}
    Therefore, to ensure $w \in H_0^1(D)$, the coefficient of this higher-order singularity must vanish identically, enforcing $C_2 = 0$.
    
    \item Next, we test the remaining homogeneous component involving $C_1$. The square integrability condition near the origin requires
    \begin{equation}\label{eq:co-energy-test}
    \int_0^\varepsilon \left( C_1 r^{-c-1} \right)^2 r^{d-1} \, dr = C_1^2 \int_0^\varepsilon r^{d-2c-3} \, dr < \infty.
    \end{equation}
    The integral \eqref{eq:co-energy-test} converges if and only if $d - 2c - 3 > -1$, which simplifies directly to the parameter range $c < \frac{d-2}{2}$.
\end{itemize}

Crucially, the condition $c < \frac{d-2}{2}$ is \textit{exactly} our fundamental assumption that guarantees the lower boundedness of the bilinear form $\form$. Consequently, unlike the case for the potential, the singular component $C_1 r^{-c}$ is indeed admissible within the energy space $H_0^1(D)$, and the constant $C_1$ cannot be assumed to be zero.

Integrating the leading derivative term $w'(r) \sim C_1 r^{-c-1}$ from $r$ to $\varepsilon > 0$, we immediately obtain the asymptotic profile of the co-potential near the origin:
\begin{equation}\label{eq:co-potential-divergence}
w(r) \sim -\frac{C_1}{c} r^{-c} \quad \text{as } r \to 0.
\end{equation}

Note that $w$ must be non-negative as a co-potential of a positive measure, which forces $C_1 < 0$. Since $c > 0$, the expression \eqref{eq:co-potential-divergence} strictly blows up to $+\infty$ at the origin.
We therefore conclude that the $1$-co-potential is not uniformly bounded, establishing that $\widehat{U}_1\nu \notin L^\infty(D)$.

\section{Stollmann--Voigt's inequality for semi-Dirichlet forms}\label{sec_SVineq}
In this section, we establish a non-symmetric version of Stollmann--Voigt's inequality. Throughout this section, we consider a regular semi-Dirichlet form 
$(\form, D(\form))$ on $L^2(E)$ with the lower bound $-\beta_0\leq0$ admitting a core $\mathscr{C}$. Let $\mathbb{M}=(\Omega, (X_t)_{t\geq0}, (\mathbb{P}_x)_{x\in E})$ be the associated Hunt process with $(\form, D(\form))$ (see \cite[\S3.3]{O13}). 
We only assume the weak sector condition $(\form.2)'$ for 
$(\form, D(\form))$.

\begin{theorem}[Stollmann--Voigt inequality] \label{thm:StollmannVoigt} 
  For any $\alpha>\beta_0$, $\mu\in {\sf S}_{0}$ and $f\in D(\form)$, we have 
 \begin{align*} 
 \int_Ef^2d\mu\leq K_{\alpha}^2\|U_{\alpha}\mu\|_{\infty}^{1/2} \|\widehat{U}_{\alpha}\mu\|_{\infty}^{1/2} \form_{\alpha}(f,f). 
 \end{align*} 
\end{theorem}

\begin{proof}
Without loss of generality, we may assume that $\mu\in {\sf S}_{0}$ satisfies $U_{\alpha} \mu,\ \widehat{U}_{\alpha}\mu \in L^\infty(E)$. By using the monotone convergence theorem, it is enough to show the inequality 
for a bounded nonnegative $f\in D(\form)$. Indeed, if the assertion 
holds for such a bounded nonnegative element in $\dom$, 
it holds for $|f|\land n\in \dom$ for 
all $n\in\mathbb{N}$ and $f\in\dom$. By $(\mathscr{E}.4)$ and \cite[(1.1.12)]{O13}, 
$\mathscr{E}_{\alpha}(|f|\land n,|f|\land n)\leq K_{\alpha}^2\mathscr{E}_{\alpha}(f,f)$, i.e., $\{|f|\land n\}_{n\in\mathbb{N}}$ is $\mathscr{E}_{\alpha}$-bounded. 
From this, we obtain the assertion for general $f\in\dom$, by way of the Banach--Saks theorem.
 For $\beta>0$, we set 
$f_{\beta}:=\beta R_{\beta+\alpha}f$. We easily confirm $f_{\beta}\to f$ in $\form_{\alpha}^{1/2}$-norm and in $L^2(\mu)$ thanks to the boundedness of $f$.

By the preparations in Section \ref{sec_preparation}, for $\mu \in \mathsf{S}_0$, there uniquely exists the associated PCAF ${\sf A}^\mu$ of $\mathbb{M}$. 
We set $R_{\alpha}(h\mu)(x):=\mathbb{E}_x\left[\int_0^{\infty}e^{-\alpha s}h(X_s)d{\sf A}_s^{\mu} \right]$ for  $h\in\mathscr{B}(E)_+$. Then $R_{\alpha}\mu$ is a quasi-continuous $\mathfrak{m}$-version of $U_{\alpha}\mu$. Moreover, the probabilistic potential $R_{\alpha}(h\mu)$ can be pointwise bounded by the Cauchy-Schwarz inequality:
\begin{align}
\left(R_{\alpha}(h\mu)(x)\right)^2&=\left(\mathbb{E}_x\left[\int_0^{\infty}e^{-\alpha s}h(X_s)d{\sf A}_s^{\mu} \right]\right)^2\nonumber\\
&\leq \mathbb{E}_x\left[\int_0^{\infty}e^{-\alpha s}h^2(X_s)d{\sf A}_s^{\mu} \right] \mathbb{E}_x\left[\int_0^{\infty}e^{-\alpha s}d{\sf A}_s^{\mu} \right] \label{eq:3.1}
\\
&= R_{\alpha}(h^2\mu)(x)R_{\alpha}\mu(x).\nonumber
\end{align}
Now, noting that $R_\alpha(\beta(f - \beta R_{\beta+\alpha}f)) = \beta R_{\beta+\alpha}f = f_\beta$, we set $g_\beta := \beta(f - f_\beta)$ so that $f_\beta = R_\alpha g_\beta$. To estimate the $L^2(\mu)$-norm of $R_\alpha g_\beta$, we use the duality expression:
\begin{align*}
\label{eq:1}
\int_E f_\beta^2 d\mu &= \int_E (R_\alpha g_\beta)^2 d\mu = 
\sup\left\{\left. \left(\int_E R_\alpha g_\beta\, \phi\, d\mu \right)^2\;\right|\;
\phi\in L^2(\mu)\cap L^{\infty}(\mu),\|\phi\|_{L^2(\mu)} \le 1 
\right\}. 
\end{align*}
By the weak sector condition $(\form.2)'$,
the energy coincidence for potentials (Proposition \ref{prop:S_0-properties}),
and the monotonicity of potentials, we have for any
$\phi \in L^2(\mu) \cap L^{\infty}(\mu)$ with $\|\phi\|_{L^2(\mu)} \le 1$:
\begin{align*}
\left( \int_E R_\alpha g_\beta \,\phi\, d\mu \right)^2&\leq 
\left( \int_E R_\alpha g_\beta \,|\phi|\, d\mu \right)^2\\
&= \form_\alpha(R_\alpha g_\beta, \widehat{U}_\alpha(|\phi|\,\mu))^2 \\
&\le K_\alpha^2 \form_\alpha(R_\alpha g_\beta, R_\alpha g_\beta) \form_\alpha(U_\alpha(|\phi|\,\mu), U_\alpha(|\phi|\,\mu)) \\[5pt]
&= K_\alpha^2 (g_\beta, R_\alpha g_\beta) \langle |\phi|\,\mu, R_\alpha(|\phi|\,\mu) \rangle.
\end{align*}
For the last factor on the right-hand side, applying the Schwarz inequality \eqref{eq:3.1} with $h=|\phi|$ and utilizing the duality and the boundedness of the potentials, we can further bound it as follows:
\begin{align*}
\langle |\phi|\,\mu, R_\alpha(|\phi|\,\mu) \rangle
&\le \left( \int_E (R_\alpha(|\phi|\,\mu))^2 d\mu \right)^{1/2} \left( \int_E \phi^2 d\mu \right)^{1/2} \\
&\le \left( \int_E R_\alpha(\phi^2 \mu) R_\alpha\mu d\mu \right)^{1/2} \cdot 1 \\
&\le \|U_\alpha\mu\|^{1/2}_\infty  \left( \int_E \phi^2  \widehat{U}_\alpha\mu d\mu \right)^{1/2} \\
&\le \|U_\alpha\mu\|^{1/2}_\infty \|\widehat{U}_\alpha\mu\|^{1/2}_\infty.
\end{align*}
Therefore, we find that
$$ \int_E f_\beta^2 d\mu \le K_\alpha^2 \|U_\alpha\mu\|^{1/2}_\infty \|\widehat{U}_\alpha\mu\|^{1/2}_\infty (g_\beta, R_\alpha g_\beta) $$
holds for any $\beta > 0$.

Since $f_{\beta}\to f$ in $L^2(\mu)$ as $\beta\to\infty$, the left-hand side converges to $\int_Ef^2d\mu$. On the other hand, the last factor in the right-hand side is equal to $\beta(f-f_{\beta},f_{\beta})$, which converges to $\form_{\alpha}(f,f)$ by 
\cite[Theorem~1.1.4(ii)]{O13} for the form $\form_{\alpha}$ (not for $\form$). 
Note that \cite[Theorem~1.1.4(ii)]{O13} remains valid under $(\form.2)'$.  
\end{proof}

  As prepared in Section \ref{sec_preparation}, if the dual form $(\widehat{\form}, \dom)$
  also satisfies the Markov property, we can directly use
  the probabilistic representations $R_{\alpha} \mu$ and $\widehat{R}_{\alpha} \mu$ 
  for any smooth measure $\mu \in \mathsf{S}$.
  This allows us to state the following corollary.

\begin{cor}\label{cor:nonSymDirStollmannVoigt}
Suppose further that $(\form.4)$ is satisfied for $(\widehat{\form},D(\form))$. For $\mu\in{\sf S}$ and $f\in D(\form)$, we have that for $\alpha>\beta_0$
\begin{align*}
\int_Ef^2d\mu\leq K_{\alpha}^2\|R_{\alpha}\mu\|_{\infty}^{1/2}\|\widehat{R}_{\alpha}\mu\|_{\infty}^{1/2}\form_\alpha(f,f).
\end{align*}
\end{cor}
\begin{proof}
Since there exists a common generalized nest $\{K_n\}$ of compact sets such that $\mathds{1}_{K_n}\mu\in {\sf S}_{0}$ and $U_{\alpha}(\mathds{1}_{K_{n}}\mu),\  \widehat{U}_{\alpha}(\mathds{1}_{K_n}\mu) \in L^\infty(E)$, we have from Theorem \ref{thm:StollmannVoigt} that
\begin{align*}
\int_Ef^2d\mu&= \lim_{n\to \infty} \int_{K_n}f^2d\mu\\
& \leq \varlimsup_{n\to \infty} K_{\alpha}^2\|R_{\alpha}(\mathds{1}_{K_n}\mu)\|_{\infty}^{1/2} \|\widehat{R}_{\alpha}(\mathds{1}_{K_n}\mu)\|_{\infty}^{1/2}\form_\alpha(f,f)\\
&\leq K_{\alpha}^2\|R_{\alpha}\mu\|_{\infty}^{1/2} \|\widehat{R}_{\alpha}\mu\|_{\infty}^{1/2}\form_\alpha(f,f).
\end{align*}
\end{proof}

In what follows in this section,
we assume the strong sector condition $(\form.2)$. 
Thanks to $(\form.2)$, we can define the notion of extended Dirichlet space 
$(\form, \dom_{e})$ (see \cite[p.~18]{O13} for the extended Dirichlet space). 
Moreover, we assume the transience of $(\form,D(\form))$
(see \cite[(1.3.7)]{O13} for the transience).  
In this case, we can consider the $0$-order version of measures of finite energy integrals:
A positive Radon measure $\nu$ on $E$ is said to be a
{\it measure of finite {\rm(}$0$-order{\rm)} energy integral} 
(denoted by $\nu\in {\sf S}_0^{(0)}$) 
if the inequality \eqref{eq:S0-def} holds with $\form_{\beta}$
on the right-hand side replaced by $\form$. Then the equation \eqref{S_0} 
with $\beta=0$ uniquely determines the functions   
$U\nu,\, \widehat{U}\nu\in D(\form)_{e}$.
The function $U\nu$ (resp.~$\widehat{U}\nu$) is called the
{\it {\rm(}$0$-order{\rm)} potential {\rm(}resp.~co-potential{\rm)} of $\nu$}.

\begin{cor}\label{cor:StollmannVoigt}
  Assume $\beta_0=0$. For any $\mu\in {\sf S}_{0}^{(0)}$ and $f\in D(\form)_e$, we have
\begin{align*}
\int_Ef^2d\mu\leq K^2\|U\mu\|_{\infty}^{1/2}\|\widehat{U}\mu\|_{\infty}^{1/2}\form(f,f).
\end{align*}
If we further assume that $(\widehat{\form},D(\form))$ satisfies $(\form.4)$, then for any $\mu\in {\sf S}$ and $f\in D(\form)_e$, we have
\begin{align*}
\int_Ef^2d\mu\leq K^2\|R\mu\|_{\infty}^{1/2}\|\widehat{R}\mu\|_{\infty}^{1/2}\form(f,f),
\end{align*}
where $R\mu(x)=\mathbb{E}_x[{\sf A}_{\infty}^{\mu}]$ and $\widehat{R}\mu(x)=\widehat{\mathbb{E}}_x[\widehat{\sf A}_{\infty}^{\mu}]$.
\end{cor}
\begin{proof}
The proof is very similar to the proof of Theorem~\ref{thm:StollmannVoigt}. We omit the details.
\end{proof}

\begin{rem}\label{rem:fitzsimmons-comparison}
In the setting of semi-Dirichlet forms, Fitzsimmons \cite{F01} established the Stollmann--Voigt inequality for any smooth measure $\mu \in \mathsf{S}$, using only the bounded co-potential $\widehat{U}_\alpha\mu$.  
For any $\alpha$-excessive function $u \in \dom$, by use of It\^o's formula, he proves the inequality with the constant $2$:
\begin{equation}
\int_E u^2 \, d\mu \le 2\|\widehat{U}_\alpha\mu\|_{\infty} \form_\alpha(u,u).\label{eq:Fit1}
\end{equation}
To extend this to a general function $f \in \dom$, one dominates $|f| \in \dom$ by its $\alpha$-reduced function $u_{|f|}$, which is an $\alpha$-excessive element in 
$\mathscr{L}_{|f|}:=\{v\in\dom\mid \tilde{v}\geq|\tilde{f}|\, {\rm q.e.}\}$
such that 
$\mathscr{E}_{\alpha}( u_{|f|},v-u_{|f|})\geq0$ for any $v\in \mathscr{L}_{|f|}$ (see \cite[(2.4.1),(2.4.2)]{O13}  and note that they remain valid for any functions in $\dom$ (in this case $|f|\in\dom$)).
Under the non-symmetric weak sector condition, applying \cite[(2.4.5)]{O13}, we have 
\begin{align*}
\int_E f^2\, d\mu\leq \int_E u_{|f|}^2\, d\mu
\stackrel{\eqref{eq:Fit1}}{\le}2\|\widehat{U}_{\alpha}\mu\|_{\infty}\form_{\alpha}(u_{|f|},u_{|f|})\le 
2K_\alpha^2\|\widehat{U}_{\alpha}\mu\|_{\infty} \form_\alpha(|f|,|f|).
\end{align*} 
On the other hand, $\form_\alpha(|f|,|f|)\leq 4K_\alpha^2 \form_\alpha(f,f)$ by \cite[Proposition~2.9]{F01}, resulting in Fitzsimmons' original inequality for general functions:  for $\alpha>\beta_0$
\begin{align*}
  \int_E f^2 \, d\mu \le 8K_\alpha^4 \|\widehat{U}_\alpha\mu\|_{\infty} \form_\alpha(f,f).
\end{align*}
However, we always have $\mathscr{E}_{\alpha}(|f|,|f|)\leq \mathscr{E}_{\alpha}(f,f)$ for any $f\in \dom$ by 
\cite[(1.1.12)]{O13}. 
One can obtain, for $\alpha>\beta_0$
\begin{align*}
  \int_E f^2 \, d\mu \le 2K_\alpha^2 \|\widehat{U}_\alpha\mu\|_{\infty} \form_\alpha(f,f).
\end{align*}
\end{rem}

\section{Closed Forms Constructed from small Perturbations}
\label{sec_smallpertub}

In this section, we consider perturbations of lower bounded closed forms and compare the class $\mathsf{S}_0$ before the perturbation with that after the perturbation.  To this end, we start with a closed form $(\form^0, D(\form^0))$ on $L^2(E)$ with a lower bound $-\beta_0\leq 0$. 
We assume $(\form^0, D(\form^0))$ admits a core $\mathscr{C} \subset D(\form^0) \cap C_0(E)$. 

Let $b$ be a bilinear form defined on ${\mathscr{C}}\times{\mathscr{C}}$. Assume that there exist {
  $0< \theta <1$} and $C\ge 0$ such that for any $u, v \in {\mathscr{C}}$,
\begin{equation} \label{perturb}
\begin{array}{rl}
  \big|b(u,v)\big| \le & {
                         \theta} \sqrt{\form^0_{\beta_0}(u,u)}\sqrt{\form^0_{\beta_0}(v,v)}  \\[8pt]
& \quad + C\Big( \sqrt{\form^0_{\beta_0}(u,u)} \|v\|_{L^2} + \|u\|_{L^2} \sqrt{\form^0_{\beta_0}(v,v)} + \|u\|_{L^2}  \|v\|_{L^2}  \Big).
\end{array}
\end{equation}
We define the perturbed bilinear form $\form$ by
$$
\form(u,v):=\form^0(u,v)+b(u,v), \quad u,v \in {\mathscr{C}}.
$$
By setting $\dom:=D(\form^0)$,  we see that the pair $(\form, \dom)$ produces a lower bounded closed form on $L^2(E)$ with a parameter
$$
\beta':=
\frac{\beta_0(1-\theta)+C(1-\theta)+C^2}{1-\theta} \  > \ \beta_0.
$$

In fact, the lower boundedness and the non-negative definiteness of $\form_{\beta'}$ can be verified as follows. 
For any $u \in \mathscr{C}$, by considering the diagonal component of \eqref{perturb}, we have the lower estimate
\begin{align*}
\form(u,u) & = \form^0(u,u) + b(u,u) \\
           &
             \ge \form^0(u,u) - \theta \form^0_{\beta_0}(u,u) - C \left( 2\sqrt{\form^0_{\beta_0}(u,u)}\|u\|_{L^2} + \|u\|_{L^2}^2 \right).
\end{align*}
By using the definition $\form^0(u,u) = \form^0_{\beta_0}(u,u) - \beta_0 \|u\|_{L^2}^2$, we can rewrite the right-hand side as:
\begin{equation} \label{eq:lower-1}
    \form(u,u) \ge (1-\theta)\form^0_{\beta_0}(u,u) - \beta_0 \|u\|_{L^2}^2 - 2C \sqrt{\form^0_{\beta_0}(u,u)}\|u\|_{L^2} - C\|u\|_{L^2}^2.
\end{equation}
To control the cross term, we apply Young's inequality $2xy \le \epsilon x^2 + \frac{1}{\epsilon}y^2$ with 
  $\epsilon = 1-\theta > 0$:
\begin{equation*}
2C \sqrt{\form^0_{\beta_0}(u,u)}\|u\|_{L^2} \le (1-\theta)\form^0_{\beta_0}(u,u) + \frac{C^2}{1-\theta}\|u\|_{L^2}^2.
\end{equation*}
Substituting this back into \eqref{eq:lower-1} directly  cancels out the primary energy term
  $(1-\theta)\form^0_{\beta_0}(u,u)$, yielding:
\begin{align*}
\form(u,u) & \ge (1-\theta)\form^0_{\beta_0}(u,u) - \beta_0 \|u\|_{L^2}^2 - \left( (1-\theta)\form^0_{\beta_0}(u,u) + \frac{C^2}{1-\theta}\|u\|_{L^2}^2 \right) - C\|u\|_{L^2}^2 \\
& = - \left( \beta_0 + C + \frac{C^2}{1-\theta} \right) \|u\|_{L^2}^2 \\
& = - \left( \frac{\beta_0(1-\theta) + C(1-\theta) + C^2}{1-\theta} \right) \|u\|_{L^2}^2
\end{align*}
which immediately implies $\form_{\beta'}(u,u) \ge 0$ for all $u \in \mathscr{C}$. By the denseness of core \(\mathscr{C}\), this non-negative definiteness extends to the whole domain $\dom = D(\form^0)$.

We assume the Markov property $(\form.4)$, that is, $u^{\#} := u \wedge a \in \dom$ and
\begin{equation}\label{eq:markov-condition}
\form(u^{\#}, u - u^{\#}) = \form^0(u^{\#}, u - u^{\#}) + b(u^{\#}, u - u^{\#}) \ge 0
\end{equation}
holds for all $u \in \dom$ and \(a\geq 0\). Under this constraint, $(\form, \dom)$ becomes a lower bounded semi-Dirichlet form on $L^2(E)$. It is crucial to note here that while the total pair $(\form, \dom)$ satisfies the Markov property, the baseline form $(\form^0, D(\form^0))$ itself is not necessarily a semi-Dirichlet form.

Nevertheless, as remarked before,  the lack of the Markov property for $(\form^0, D(\form^0))$ poses no obstruction to the potential-theoretic formulation of the smooth measures. Since both $(\form^0, D(\form^0))$ and $(\form, \dom)$ are lower bounded closed forms on $L^2(E)$, we can legitimately define the class $\mathsf{S}_0(\form)$ of measures of finite energy integrals directly with respect to the perturbed form $\form$, independently of whether the Markov property holds. Specifically, a positive Radon measure $\mu$ belongs to $\mathsf{S}_0(\form)$ if
{
for any $\beta > \beta'$, 
}
there exists a constant $C > 0$ such that
\begin{equation}\label{eq:S0-perturbed-def}
\int_E |u(x)| \, \mu(dx) \le C \sqrt{\form_{\beta}(u,u)} \quad \text{for all } u \in \mathscr{C},
\end{equation}
and the unique existence of the associated potential and co-potential in $\dom$ is fully guaranteed by the Hilbert space structure.
Then we have the following proposition.
 
\begin{prop} \label{prop:S0-coincide}
The two classes of positive Radon measures of finite energy integrals coincide. Namely, we have
\begin{equation}
\mathsf{S}_0(\form^0) = \mathsf{S}_0(\form).
\end{equation}
\end{prop}
\begin{proof}
This equivalence is an immediate consequence of the norm equivalence between the baseline form and the perturbed form. Indeed, by the perturbation condition \eqref{perturb} and Young's inequality, for any $\beta > \beta'$, there exist constants $c_1, c_2 > 0$ such that
\begin{equation}\label{eq:norm-equiv}
c_1 \sqrt{\form^0_{\beta}(u,u)} \le \sqrt{\form_{\beta}(u,u)} \le c_2 \sqrt{\form^0_{\beta}(u,u)}
\end{equation}
holds for all $u \in \mathscr{C}$.

Now, let $\mu \in \mathsf{S}_0(\form^0)$. By definition \eqref{eq:S0-def}, for $\beta > \beta' (> \beta_0)$, there exists a constant $C > 0$ satisfying $\int_E |u| \, d\mu \le C \sqrt{\form^0_{\beta}(u,u)}$ for all $u \in \mathscr{C}$. Applying the left-hand inequality of \eqref{eq:norm-equiv}, we directly obtain
\begin{equation*}
\int_E |u| \, d\mu \le \frac{C}{c_1} \sqrt{\form_{\beta}(u,u)} \quad \text{for all } u \in \mathscr{C},
\end{equation*}
which implies $\mu \in \mathsf{S}_0(\form)$, and hence $\mathsf{S}_0(\form^0) \subset \mathsf{S}_0(\form)$. Conversely, if $\mu \in \mathsf{S}_0(\form)$, the right-hand inequality of \eqref{eq:norm-equiv} similarly yields $\int_E |u| \, d\mu \le C c_2 \sqrt{\form^0_{\beta}(u,u)}$, proving the reverse inclusion $\mathsf{S}_0(\form) \subset \mathsf{S}_0(\form^0)$. By the denseness of the core $\mathscr{C}$, the proof is complete.
\end{proof}

In applications, many non-symmetric processes are naturally obtained through certain transformations of symmetric ones. With this in mind, 
 we assume that $\form^0$ is symmetric, i.e., $\form^0(u,v)=\form^0(v,u)$, in the rest of this section.   Then, for each measure $\nu \in \mathsf{S}_0 := \mathsf{S}_0(\form^0) = \mathsf{S}_0(\form)$ and any parameter $\beta > \beta'$, although both the potentials $U_{\beta} \nu$ and the co-potentials $\widehat{U}_{\beta} \nu$ belong to the same domain $\dom$, the non-symmetry of the perturbation $b(u,v)$ induces a significant structural discrepancy between them.

For a measure $\nu \in \mathsf{S}_0$,  potentials $U_\beta \nu \in \dom$  and $\widehat{U}_\beta \nu \in \dom$ 
are expresses as 
$$
\form_\beta(U_\beta \nu, v) =\form_\beta^0(U_\beta \nu, v) +b(U_\beta \nu, v)=\int_E v(x) \nu(dx) \quad {\rm for \ all} \ v\in {\mathscr{C}}
$$
and 
$$
\form_\beta(v, \widehat{U}_\beta \nu) =\form_\beta^0(v, \widehat{U}_\beta \nu) +b(v, \widehat{U}_\beta \nu)=\int_E v(x) \nu(dx) 
\quad {\rm for \ all} \ v\in {\mathscr{C}}.
$$
Then, subtracting the both sides and putting $w:=U_\beta \nu -\widehat{U}_\beta \nu$, we see that 
$$
\form_\beta^0(w,v) =  b(v, \widehat{U}_\beta \nu)-b(U_\beta \nu, v)  \quad {\rm for \ all} \ v\in {\mathscr{C}}.
$$
By the denseness of the core \(\mathscr{C}\), we can substitute $v = w=U_\beta \nu -\widehat{U}_\beta \nu$ into the above equality,  we can conclude the following  energy formula holds.

\begin{prop} \rm Assume that $\form^0$ is symmetric. For $\nu \in \mathsf{S}_0$, the following holds for \(\beta > \beta'\).
\begin{equation} \label{enq:energy-formula}
\form_\beta(U_\beta \nu -\widehat{U}_\beta \nu, U_\beta \nu -\widehat{U}_\beta \nu)
= b(U_\beta \nu, \widehat{U}_\beta \nu) - b(\widehat{U}_\beta \nu, U_\beta \nu).
\end{equation}
\end{prop}

\begin{rem} \rm 
The identity \eqref{enq:energy-formula} provides a clear geometric insight into how the asymmetry of the system distorts the potential structure.
If the perturbation $b$ is also symmetric, i.e., $b(u,v) = b(v,u)$, the right-hand side of \eqref{enq:energy-formula} vanishes identically. Then, by the coercivity (or the sector condition) of the lower bounded closed form $\form_\beta$, it immediately follows that 
$U_\beta \nu =\widehat{U}_\beta \nu$.

Conversely, if the perturbation $b$ is non-symmetric, the antisymmetric part of $b$, resembling a Lie bracket, directly emerges on the right-hand side. This antisymmetric component acts as the exact quantitative source of energy that breaks the symmetry, forcing a profound structural discrepancy between the potentials $U_\beta \nu$ and $\widehat{U}_\beta \nu$.
\end{rem}

\begin{rem} \rm 

Regarding the construction of lower bounded semi-Dirichlet forms through non-symmetric perturbations satisfying the condition \eqref{perturb}, we refer the reader to \cite{U26} for various other fundamental examples. In particular, \cite{U26} provides concrete constructions where a symmetric diffusion process is perturbed by a non-symmetric non-local form, as well as examples where a symmetric non-local form is perturbed by another non-symmetric non-local form. 
\end{rem}

For any \(\mu \in \mathsf{S}_0(\form^0) = \mathsf{S}_0(\form) \) and \(\beta >\beta_0\), there exists \(U_\beta^0 \mu \in D(\form^0)=\dom\) such that
  \[\form_\beta^0(U_\beta^0 \mu, v) =\form_\beta^0(v,U_\beta^0 \mu) = \int_E v(x) \mu(dx) \quad {\rm for \ all} \ v\in {\mathscr{C}}.\]
As the same way as (\ref{eq:metric-def}), we define \(\rho_\beta^0\) by
\[\rho_\beta^0(\mu, \nu) := \sqrt{\form_\beta^0 (U_\beta^0 \mu - U_\beta^0 \nu, U_\beta^0 \mu - U_\beta^0 \nu)}.\]
Then we have the following proposition.
\begin{prop}\label{prop:homeo}
Assume that $\form^0$ is symmetric. Then, for \(\beta >\beta'\), there exists a constant \(C>0\) such that, for \(\mu, \nu \in \mathsf{S}_0(\form^0) = \mathsf{S}_0(\form)\), it holds that
\[|\rho_\beta^0(\mu, \nu) - \rho_\beta (\mu, \nu)| \leq C \rho_\beta^0(\mu, \nu) \wedge \rho_\beta(\mu, \nu).\]
In particular, \(( \mathsf{S}_0(\form), \rho_\beta)\) is homeomorphic to \((\mathsf{S}_0(\form^0), \rho_\beta^0)\).
\end{prop}
\begin{proof}
We take \(\beta > \beta'\) and \(\mu, \nu \in \mathsf{S}_0(\form^0) = \mathsf{S}_0(\form)\). We set \(\xi := \mu-\nu\) and we may assume that 
\(\xi \not = 0\) without loss of generality.  {
  By taking approximating sequences $\{v_n\}, \{w_n\} \subset {\mathscr C}$ such that 
$v_n \to U_\beta \xi$ both in $\form_\beta$ and in $\form^0_\beta$, and $w_n \to U^0_\beta \xi$ both in $\form_\beta$ and in $\form^0_\beta$
we evaluate the distance.  Using the continuity of the forms, the definition of the potentials, and the symmetry of $\form^0$}, we have
\begin{align*}
 \rho_\beta(\mu, \nu)^2 &= \form_\beta(U_\beta\xi, U_\beta\xi) \  =\ 
                          \lim_{n\to \infty} \form_\beta(U_\beta\xi,v_n) = \lim_{n\to \infty}   \int_E v_n \,d\xi  \\
                        & =  
                          \lim_{n\to \infty} \form_\beta^0(U_\beta^0 \xi,v_n) \ = \ \lim_{n\to \infty}  \form_\beta^0(v_n, U_\beta^0 \xi)  = \form_\beta^0(U_\beta \xi,  U_\beta^0 \xi)   \\
                        &  = \form_\beta(U_\beta\xi, U_\beta^0\xi) - b(U_\beta\xi, U_\beta^0\xi)  \  =  \ 
  \lim_{n\to \infty}  \form_\beta(U_\beta\xi, w_n) - b(U_\beta\xi, U_\beta^0\xi) \\
                        & = 
                          \lim_{n\to \infty} \int_E w_n \, d\xi - b(U_\beta\xi, U_\beta^0\xi)  =
\lim_{n\to \infty}  \form_\beta^0 (U_\beta^0 \xi, w_n)
- b(U_\beta\xi, U_\beta^0\xi)    \\
                        & =
                          \form_\beta^0 (U_\beta^0 \xi, U_\beta^0 \xi) - b(U_\beta\xi, U_\beta^0\xi)    \ = \  \rho_\beta^0(\mu, \nu)^2 -b(U_\beta\xi, U_\beta^0\xi).
\end{align*}
Combining this with (\ref{perturb}), (\ref{eq:norm-equiv}) and \(\beta > \beta' > \beta_0\), we have
\begin{align*}
& \!\!\! |\rho_\beta^0(\mu, \nu)- \rho_\beta(\mu, \nu)| = \frac{1}{\rho_\beta^0(\mu, \nu)+ \rho_\beta(\mu, \nu)}|b(U_\beta\xi, U_\beta^0\xi)|\\
&\leq  \frac{C}{\rho_\beta^0(\mu, \nu)+ \rho_\beta(\mu, \nu)} \left( \sqrt{\form^0_{\beta_0}(U_\beta\xi,U_\beta\xi)} + \|U_\beta\xi\|_{L^2}\right) \left( \sqrt{\form^0_{\beta_0}(U_\beta^0\xi,U_\beta^0\xi)} + \|U_\beta^0\xi\|_{L^2}\right)\\
&\leq  \frac{C}{\rho_\beta^0(\mu, \nu)+ \rho_\beta(\mu, \nu)} \left(\frac{1}{c_1} \sqrt{\form_{\beta'}(U_\beta\xi,U_\beta\xi)} + \|U_\beta\xi\|_{L^2}\right) \left( \sqrt{\form^0_{\beta_0}(U_\beta^0\xi,U_\beta^0\xi)} + \|U_\beta^0\xi\|_{L^2}\right)\\
&\leq  \frac{C}{\rho_\beta^0(\mu, \nu)+ \rho_\beta(\mu, \nu)}\rho_\beta(\mu, \nu) \rho_\beta^0(\mu, \nu)  
\ \leq  \ C \rho_\beta^0(\mu, \nu) \wedge \rho_\beta(\mu, \nu).
\end{align*}
Here, a constant \(C\) may change from line to line.
\end{proof}

\subsection{Examples: Perturbations of Diffusion Processes on open sets}
 In this subsection, we present three examples and show that, when the perturbation is sufficiently small, the class of measures of finite energy integrals together with its topology coincides with that of the unperturbed symmetric form. Throughout all three examples, the underlying unperturbed symmetric form \((\form^0, D(\form^0))\) is taken to be the same and is given by the following form. Let $D\subset \real^d$ be an open set \ $(d\ge 2)$.  Assume that $A(x)=(a_{ij}(x))$ is a $d\times d$ symmetric matrix valued measurable function on $D$ satisfying that there exist $0<\lambda \le  \Lambda$ such that 
$$
\lambda |\xi|^2 \le A(x) \xi \cdot \xi \le \Lambda |\xi|^2, \quad  x\in D, \ \xi \in \real^d.
$$
Define a symmetric form $\form^0$ on $C_0^\infty(D) \times C_0^\infty(D)$ as follows:
$$
\form^0(u,v):= \int_D A(x) \nabla u(x) \cdot \nabla v(x)dx.
$$
Note that the form $\form^0$ satisfies the inequality: 
$$ 
\lambda {\mathbb D}(u,u) \le \form^0(u,u) \le \Lambda {\mathbb D}(u,u)
$$
for  $u \in C_0^\infty(D)$, where ${\mathbb D}(u,v) = \int_D \nabla u(x) \cdot \nabla v(x) dx 
$ is the Dirichlet integral on $D$.  Then this implies that $(\form^0, H^1_0(D))$ is a regular symmetric strongly local Dirichlet form on $L^2(D):=L^2(D;dx)$.

We construct several perturbed semi-Dirichlet forms of the type $\form = \form^0 + b$, where the non-symmetric perturbation $b$ is formally given by the following bilinear form:
\begin{equation}\label{eq:b-form-def}
b(u,v) := \int_D B(x) \cdot \nabla u(x) \, v(x) \, \mu(dx).
\end{equation}
Here, $B(x) = (B_i(x))$ is a drift vector field  and $\mu$ is a Borel measure that may be mutually singular to the Lebesgue measure $dx$ on $D$, but satisfies our key inequality \eqref{perturb}.

\begin{exam}[Singular Local Drifts on the Lebesgue Measure] \label{ex:coincides_1}
\rm 
Let $D \subset \mathbb{R}^d$ ($d \ge 3$) be an open set. We consider a non-symmetric perturbation given by a singular vector field $B(x)$ with respect to the Lebesgue measure $dx$:
\begin{align*}
b(u, v) := \int_D B(x) \cdot \nabla u(x) \,  v(x) dx.
\end{align*}
We assume that the vector field satisfies either $|B|^2 \in \mathscr{K}_d(D)$ or $B \in (L^q(D))^d$ for $d \le q \le\infty.$    Here \(\mathscr{K}_d(D):=\{f\in L^0(D) \mid \lim_{r\searrow 0}\sup_{x\in D} \int_{|x-y|\leq r}\frac{|f(y)|}{|x-y|^{d-2}}\,dy =0 \}\) is the space of Kato class functions and $L^0(D)$ denotes the space of measurable functions on $D$. We prove that the key inequality \eqref{perturb} holds for both cases.

\smallskip
\noindent
\underline{Case 1: $|B|^2 \in \mathscr{K}_d(D)$:}

\smallskip
By the analytic characterization of the Kato class, the measure $\mu(dx) := |B(x)|^2 dx$ satisfies the condition that for any $\delta > 0$, there exists a constant $C_\delta > 0$ such that
\begin{align*}
\int_D v^2 |B(x)|^2 dx \le \delta \int_D |\nabla v|^2 dx + C_\delta \|v\|_{L^2(D)}^2, \quad \text{for all } v \in C_0^\infty(D).
\end{align*}
Applying the Cauchy--Schwarz inequality to the perturbation $b(u, v)$, we have
\begin{align*}
|b(u, v)| \le \left( \int_D |\nabla u|^2 dx \right)^{1/2} \left( \int_D v^2 |B|^2 dx \right)^{1/2}.
\end{align*}
Combined with the Kato class inequality above and the uniform ellipticity of the unperturbed symmetric form $\form^A$, we can bound the perturbation by an arbitrarily small constant $a \propto \sqrt{\delta}$ by choosing $\delta > 0$ small enough, and \eqref{perturb} holds.
\smallskip

\noindent
\underline{Case 2: $B=(B_i)_{1\leq i \leq d}  \in (L^q(D))^d$ for $d \le q \le \infty$:} 

\smallskip
Suppose that $B=(B_i)_{1\leq i \leq d}  \in (L^q(D))^d$ for $d < q \le \infty$. We take \(p\) satisfying \(p^{-1}+2q^{-1}=1\), then \(p<d/(d-2)\) if and only if $d< q \le \infty$. By the Cauchy--Schwarz inequality, we have
\begin{align*}
\varlimsup_{r\searrow 0} \sup_{x\in D} \int_{|x-y|\leq r} \frac{|B(y)|^2}{|x-y|^{d-2}}\,dy &\leq \varlimsup_{r\searrow 0} \sup_{x\in D} \|B\|_{L^q}^2 \left( \int_{|x-y|\leq r} \frac{1}{|x-y|^{p(d-2)}}\,dy \right)^{1/p}\\
& \leq  \|B\|_{L^q}^2 \varlimsup_{r\searrow 0} \left( \int_0^r \frac{\theta^{d-1}}{\theta^{p(d-2)}}\,d\theta \right)^{1/p}\\
& = 0.
\end{align*}
Hence $|B|^2 \in \mathscr{K}_d(D)$, and \eqref{perturb} holds for $d < q \le \infty$ by the Case 1.

Next we assume $B=(B_i)_{1\leq i \leq d} \in (L^d(D))^d$. By the Gagliardo--Nirenberg-Sobolev inequality, there exists \(C_G>0\) such that\(\|v\|_{L^{\frac{2d}{d-2}}(D)} \leq C_G \|\nabla v\|_{L^2(D)} \) holds for any \(v \in C_0^\infty\). We take \(M>0\) satisfying \(C_G \|B\mathds{1}_{\{|B|>M\}}\|_{L^d(D)} < 2^{-1} \lambda \). Then, by the Cauchy--Schwarz inequality, we have
\begin{align}
\nonumber \left| \int_D B(x) \mathds{1}_{\{|B|>M\}}\nabla u(x) v(x)\,dx \right| &\leq \|B\mathds{1}_{\{|B|>M\}}\|_{L^d(D)}\|\nabla u\|_{L^2(D)} \|v\|_{L^{\frac{2d}{d-2}}(D)}\\
& \leq 2^{-1}\lambda  \|\nabla u\|_{L^2(D)}\|\nabla v\|_{L^2(D)} \label{eq:q=d_1}
\end{align}
for \(u,v \in C_0^\infty\).
By the Cauchy--Schwarz inequality, we have
\begin{align}
\left| \int_D B(x) \mathds{1}_{\{|B|\leq M\}}\nabla u(x) v(x)\,dx \right| &\leq M \|\nabla u\|_{L^2(D)} \|v\|_{L^2(D)} \label{eq:q=d_2}
\end{align}
for \(u,v \in C_0^\infty\). By (\ref{eq:q=d_1}) and (\ref{eq:q=d_2}), we have 
\begin{align*}
|b(u,v)| &\leq 2^{-1}\lambda  \sqrt{{\mathbb D}(u,u)}\sqrt{{\mathbb D}(v,v)} + M\sqrt{{\mathbb D}(u,u)} \|v\|_{L^2(D)}\\
& \leq 2^{-1} \sqrt{\mathscr{E}^0(u,u)}\sqrt{\mathscr{E}^0(v,v)}+ M \lambda^{-1/2}\sqrt{\mathscr{E}^0(u,u)} \|v\|_{L^2(D)}.
\end{align*}

In both cases, the perturbed form $\form(u, v) = \form^A(u, v) + b(u, v)$ is well-defined and satisfies the lower boundedness. The framework of Case 2 further extends Case 1 by allowing the drift $B(x)$ to possess severe local singularities that do not necessarily belong to the standard $L^d(D)$ space, while precisely preserving the semi-Dirichlet properties of the diffusion process.
\end{exam}

\begin{exam}\label{ex:coincides_2} \rm 
 Let $S \subset D$ be a sphere compactly contained in $D$. Let $\sigma_S$ be the uniform surface measure concentrated on $S$. On this sphere, we activate a singular tangential flow given by the following bilinear form:
\begin{align*}
b(u, v) := \gamma\int_{S}  B(x) \cdot \nabla_S u(x) \,  v(x) \sigma_S(dx),
\end{align*}
where \(\gamma\) is a positive constant and $B(x)$ is a smooth tangential vector field on $S$ (i.e., $B(x)$ is orthogonal to the normal vector $n(x)$ at every $x \in S$). 

Since $B$ is tangential to the manifold $S$, the inner product $B \cdot \nabla u$ only involves the tangential derivative $\nabla_S u$. By applying the integration by parts on the sphere $S$, the diagonal component of the perturbation evaluates to:
\begin{align*}
b(u, u) &= \int_{S} (B \cdot \nabla_S u) u d\sigma_S = \frac{1}{2} \int_{S} B \cdot \nabla_S (u^2) d\sigma_S \\
&= -\frac{1}{2} \int_{S} (\text{div}_S B) u^2 d\sigma_S,
\end{align*}
where $\text{div}_S B$ is the surface divergence of $B$ on $S$.
If we choose $B(x)$ to be a divergence-free vector field on $S$ (e.g., a rigid rotational vector field along $S$), we have $\text{div}_S B = 0$, which implies $b(u, u) = 0$ for all $u \in C_0^\infty(D)$. Consequently, the perturbed form $\form(u, u) = \form^A(u, u) + b(u, u) = \form^A(u, u)$ automatically preserves the lower boundedness without any restriction on the magnitude of the vector field $B$.

Furthermore, by the fractional trace embedding theorem, the trace of the function $u \in H^1_0(D)$ itself belongs to $H^{1/2}(S)$, and the tangential derivative $\nabla_S u$ belongs to the dual space $H^{-1/2}(S)$. The duality pairing ensures that $|b(u, v)| \le C  \gamma \|u\|_{H^1_0(D)} \|v\|_{H^1_0(D)}$, thus satisfying the weak sector condition $(\mathscr{E},2)'$. Therefore, the pair $(\form, H^1_0(D))$ yields a legitimate non-symmetric lower bounded semi-Dirichlet form, representing a diffusion process on $D$ that experiences a singular tangential swirl exactly when it hits the spherical surface $S$.  Moreover, for small \(\gamma >0\), the inequality \eqref{perturb} holds and hence, by Proposition \ref{prop:S0-coincide} and \ref{prop:homeo}, the set of  measures of finite energy integrals are coincide \(\mathsf{S}_0(\form) = \mathsf{S}_0(\form^0)\) and \(( \mathsf{S}_0(\form), \rho_\beta)\) is homeomorphic to \((\mathsf{S}_0(\form^0), \rho_\beta^0)\).
\end{exam}

\begin{exam}[Non-local Asymmetric Perturbation on a Fractal Set]\label{ex:coincides_3} \rm 

Let $F \subset D$ be a closed subset which is an $s$-set with $d-2 < s < d$, and let $\mu = \mathscr{H}^s|_F$ be the $s$-dimensional Hausdorff measure restricted to $F$. Since the fractal set $F$ lacks a differentiable structure, the local drift $\nabla u$ cannot be traced onto $F$. Instead, we introduce a highly singular non-local asymmetric perturbation supported purely on the fractal $F$.

Let $k: F \times F \to \mathbb{R}$ be a measurable anti-symmetric kernel, i.e., $k(x, y) = -k(y, x)$. Instead of assuming $k$ to be bounded, we allow $k(x, y)$ to possess a strong singularity near the diagonal $x=y$. Let $\beta = 1 - \frac{d-s}{2} \in (0, 1)$ be the fractional smoothness index. We assume that $k$ satisfies the following weighted $L^2$-integrability condition:
\begin{align*}
\sup_{x \in F} \int_F |k(x, y)|^2 |x-y|^{s+2\beta} \mu(dy) =: M^2 < \infty.
\end{align*}
We define the non-symmetric bilinear form $b(u, v)$ by
\begin{align*}
b(u, v) := \gamma \int_F \int_F (u(x)-u(y))v(x) k(x, y) \mu(dx)\mu(dy)
\end{align*}
 for a positive constant \(\gamma>0\). Due to the anti-symmetry of the kernel $k(x, y)$, the diagonal component evaluates to:
\begin{align*}
b(u, u) &= \frac{\gamma}{2} \int_F \int_F (u(x)-u(y))(u(x)+u(y)) k(x, y) \mu(dx)\mu(dy) \\
&= \frac{\gamma}{2} \int_F \int_F (u(x)^2 - u(y)^2) k(x, y) \mu(dx)\mu(dy) = 0.
\end{align*}
Thus, on the diagonal, the identity $\form(u, u) = \form^0(u, u)$ holds, which preserves the lower boundedness without any further constraints.

Furthermore, by symmetrizing the integrand as 
$$
b(u, v) = \frac{\gamma}{2} \iint (u(x)-u(y))(v(x)+v(y)) k(x, y) \mu(dx)\mu(dy),
$$
multiplying and dividing by $|x-y|^{(s+2\beta)/2}$, and applying the Cauchy--Schwarz inequality to the double integral, we obtain
\begin{align*}
|b(u, v)| &\le \frac{\gamma}{2} \left( \int_F \int_F \frac{|u(x)-u(y)|^2}{|x-y|^{s+2\beta}} \mu(dx)\mu(dy) \right)^{1/2} \\
&\quad \times \left( \int_F \int_F (v(x)+v(y))^2 |k(x, y)|^2 |x-y|^{s+2\beta} \mu(dx)\mu(dy) \right)^{1/2} \\
&\le \gamma [u]_{B^\beta_{2,2}(F)} \left( \int_F v(x)^2 \left( \int_F |k(x, y)|^2 |x-y|^{s+2\beta} \mu(dy) \right) \mu(dx) \right)^{1/2} \\
&\le \gamma M [u]_{B^\beta_{2,2}(F)} \|v\|_{L^2(F; \mu)},
\end{align*}
where $[u]_{B^\beta_{2,2}(F)}$ is the Gagliardo semi-norm of the Besov space. 
According to the fractional trace embedding theorem by Jonsson--Wallin \cite{JW84}, $H^1_0(D)$ 
is continuously embedded into the Besov space $B^\beta_{2,2}(F)$ as well as $L^2(F; \mu)$. 
Thus, there exists a constant $C_T > 0$ such that $[u]_{B^\beta_{2,2}(F)} \le C_T \|u\|_{H^1_0(D)}$ and 
$\|v\|_{L^2(F; \mu)} \le C_T \|v\|_{H^1_0(D)}$. Consequently, we obtain $|b(u, v)| \le \gamma M C_T^2 \|u\|_{H^1_0(D)} \|v\|_{H^1_0(D)}$. 

Therefore, this singular non-local perturbation strictly satisfies the weak sector condition without requiring $\int |k|d\mu < \infty$,  yielding a lower bounded semi-Dirichlet form. Probabilistically, this corresponds to a diffusion process that experiences a highly singular asymmetric jump (such as a fractional Cauchy-type jump) restricted exclusively to the fractal set $F$.

Moreover, for small \(\gamma >0\), the inequality \eqref{perturb} holds and hence, by Proposition \ref{prop:S0-coincide} and \ref{prop:homeo}, 
  the classes of measures of finite energy integrals coincide
  (i.e., \(\mathsf{S}_0(\form) = \mathsf{S}_0(\form^0)\)) and \(( \mathsf{S}_0(\form), \rho_\beta)\) is homeomorphic to \((\mathsf{S}_0(\form^0), \rho_\beta^0)\).
\end{exam}

\begin{rem}  \label{rem4} \rm
We emphasize that for all the highly singular non-symmetric perturbations constructed in Examples \ref{ex:coincides_1} and \ref{ex:coincides_2}, \ref{ex:coincides_3} with small \(\gamma\) (as well as the Hardy-class drift investigated in Section 2),  the class of measures of finite energy integrals, together with its topology, coincides exactly with that of the unperturbed symmetric form.

This analysis illustrates a clear contrast between the energy framework and the behavior of individual potentials. Macroscopically, the singular perturbations considered here do not change the admissible class $\mathsf{S}_0$ of measures of finite energy integrals. Microscopically, however, the non-symmetry splits the dual notions of potentials; even if a measure belongs to $\mathsf{S}_0$, 
its potential can remain uniformly bounded while the co-potential diverges. 

Therefore, these examples clarify the exact geometric and analytic conditions under which this potential-theoretic asymmetry occurs within the stable framework of semi-Dirichlet forms.

\end{rem}

\section{Criterion for $\mathsf{S}_0(\form^0) \neq \mathsf{S}_0(\form)$}\label{sec_largepurturb}
As emphasized in the preceding discussions (see Remark \ref{rem4}), the highly singular perturbations considered so far firmly preserve the macroscopic energy framework, strictly maintaining that the classes coincide (i.e., $\mathsf{S}_0(\form) = \mathsf{S}_0(\form^0)$)  and \(( \mathsf{S}_0(\form), \rho_\beta)\) is homeomorphic to \((\mathsf{S}_0(\form^0), \rho_\beta^0)\) as guaranteed by Propositions \ref{prop:S0-coincide} and \ref{prop:homeo}.

In this section, we show that this robust structural invariance can be broken. We consider a lower bounded semi-Dirichlet form obtained by a non-symmetric perturbation of a base symmetric form $\form^0$. Let $(\form^0, D(\form^0))$ be a regular symmetric Dirichlet form on $L^2(E)$ and let $\mathscr{C}\subset D(\form^0)\cap C_0(E)$ 
be a core of $D(\form^0)$. Consider a non-symmetric bilinear form $b$ defined on $\mathscr{C} \times \mathscr{C}$. 
We assume that there exist  a constant $\beta_0>0$ and a positive Radon measure $\mu$ on $E$ satisfying
\begin{equation}\label{eq:break_1}
\sup_{u \in \mathscr{C}} \frac{\|u\|_{L^2(E;\mu)}^2}{\form_{\beta_0}^0(u,u)} =\infty,
\end{equation}
and there exists $c>0$ such that, for any $u\in \mathscr{C}$,
\begin{equation}\label{eq:break_2}
c \|u\|_{L^2(E;\mu)}^2 \leq b(u,u)+\beta_0 \|u\|_{L^2(E;\mathfrak{m})}^2.
\end{equation}
Assume also that the perturbed form  $\form:=\form^0+b$ satisfies the strong sector condition $(\form.2)$. Then, the condition \eqref{eq:break_2} ensures its lower boundedness $(\form.1)$, and consequently, $\form$ determines a well-defined lower bounded closed form on $L^2(E;\mathfrak{m})$ with domain $D(\form) = D(\form^0)$.

\begin{prop}\label{prop:break}
It holds that $\mathsf{S}_0(\form^0) \subsetneq \mathsf{S}_0(\form)$.
\end{prop}

\begin{proof}
By (\ref{eq:break_2}), it holds that \(\form^0(u,u) \leq \form_{\beta_0}(u,u)\) and so $\mathsf{S}_0(\form^0) \subset \mathsf{S}_0(\form)$.

By (\ref{eq:break_1}), there exists
\(h\in L^2(E;\mu)\) such that the linear functional
\[\Lambda_h(u):=\int_E u h\,d\mu ,\qquad u\in\mathscr{C}\]
is not continuous with respect to the norm $(\form_{\beta_0}^0)^{1/2}$. Indeed, suppose that $\Lambda_h$ is continuous for every $h\in L^2(E;\mu)$. For each $u\in\mathscr{C}$, we define
\[F_u:L^2(E;\mu) \ni h \mapsto \int_E uh\,d\mu \in \mathbb{R}.\]
Then $F_u$ is a bounded linear functional on
$L^2(E;\mu)$ and \(\|F_u\| = \|u\|_{L^2(E;\mu)}\). For \(B:=\{u\in\mathscr{C}:\form_{\beta_0}^0(u,u)\le 1\}\), we have
\[
\sup_{u\in B}|F_u(h)|
=
\sup_{u\in B}
\left|\int_E uh\,d\mu\right|
<\infty .
\]
Hence the family \(\{F_u\mid u\in B\}
\subset (L^2(E;\mu))^* \) is pointwise bounded. By the Banach--Steinhaus theorem,
we have \(\sup_{u\in B}
\|u\|_{L^2(E;\mu)}=\sup_{u\in B}\|F_u\| <\infty\) and this contradicts \eqref{eq:break_1}.

We take \(h\in L^2(E;\mu)\) such that \(\Lambda_h\) is not continuous. Then it holds that \(|h|\mu \not \in \mathsf{S}_0(\form^0)\). However, by (\ref{eq:break_2}), for every $u\in\mathscr{C}$,
\[\form_{\beta_0}(u,u) =  \form^0(u,u)+b(u,u)+\beta_0\|u\|_{L^2(E;\mathfrak{m})}^2 \geq c\|u\|_{L^2(E;\mu)}^2  \]
and it holds that
\[\left|\int_E u |h|\,d\mu\right| \leq \|h\|_{L^2(E;\mu)}\|u\|_{L^2(E;\mu)}
\leq \frac{\|h\|_{L^2(E;\mu)}}{\sqrt c} \sqrt{\form_{\beta_0}(u,u) }.\]
Hence we have \(|h|\mu \in \mathsf{S}_0(\form).\)
\end{proof}

\begin{exam}\rm
Consider that $D := \{|x| < R\} \subset \mathbb{R}^d$ ($d \ge 3$) is the open ball centered at the origin of radius $R>0$. We define a base symmetric form with a degenerate diffusion coefficient $\gamma \ge 0$ as follows:
\begin{align*}
\form^0(u, v) = \int_D \nabla u(x) \cdot \nabla v(x) |x|^\gamma dx, \quad u, v \in C_0^\infty(D).
\end{align*}
To consider a non-symmetric perturbation that yields a failure of the standard norm equivalence, we define an inward-pointing singular drift $B(x) = -c \frac{x}{|x|}$ ($c>0$) and pair it with a singular killing potential. For a parameter $\delta$ satisfying $0<\delta<d-1$, we specify the perturbation as
\begin{align*}
b(u, v)  & := \int_D B(x) \cdot \nabla u(x) \,  v(x) |x|^{-\delta} dx + \int_D u(x) v(x) |x|^{-(2\delta+\gamma)} dx  \\
& = -c \int_D  x \cdot \nabla u(x) \,  v(x) |x|^{-\delta-1} dx + \int_D u(x) v(x) |x|^{-(2\delta+\gamma)} dx.
\end{align*}
By denoting the drift as $\tilde{B}(x) = B(x)|x|^{-\delta}$, its divergence is $\text{div}\tilde{B} = -c(d-1-\delta)|x|^{-1-\delta}$. The diagonal component evaluates to
\begin{align*}
b(u, u) = \frac{c(d-1-\delta)}{2} \int_D u(x)^2 |x|^{-1-\delta} dx + \int_D u(x)^2 |x|^{-(2\delta+\gamma)} dx.
\end{align*}
Since $\delta < d-1$, both terms are strictly positive, ensuring the lower boundedness $b(u, u) \ge 0$.  Furthermore, the truncation 
$u^\# = (0 \vee u) \wedge 1$ yields $b(u^\#, u-u^\#) \ge 0$ owing to the disjoint gradient supports and the positivity of the potential. 
Thus, the Markov property is satisfied. 

The accompanied potential $|x|^{-(2\delta+\gamma)}$ absorbs the non-symmetric drift, ensuring that the perturbed form $\form = \form^0 + b$ satisfies the weak sector condition.  Therefore, $\form$ is a well-defined semi-Dirichlet form for any $\delta \in (0, d-1)$.
Indeed, we explicitly verify the sector condition $(\form.2)$ to demonstrate how the accompanied singular potential $|x|^{-(2\delta+\gamma)}$ controls the asymmetric drift. Since $|B(x)| = c$, applying the Cauchy--Schwarz inequality to the drift term yields
\begin{align*}
\left| \int_D B(x) \cdot \nabla u(x) \, v(x) |x|^{-\delta} dx \right| 
&\le c \int_D |\nabla u(x)| \, |x|^{\gamma/2}  |v(x)| \, |x|^{-\delta-\gamma/2} dx \\
&\le c \left( \int_D |\nabla u(x)|^2 |x|^\gamma dx \right)^{1/2} \left( \int_D v(x)^2 |x|^{-(2\delta+\gamma)} dx \right)^{1/2} \\
&\le c \form^0(u, u)^{1/2} b(v, v)^{1/2} \\
&\le c \form(u, u)^{1/2} \form(v, v)^{1/2}.
\end{align*}
As for the potential term, it is similarly bounded by applying the Cauchy--Schwarz inequality:
\begin{align*}
\int_D |u(x)v(x)|  \, |x|^{-(2\delta+\gamma)} dx 
&\le \left( \int_D u(x)^2 |x|^{-(2\delta+\gamma)} dx \right)^{1/2} \left( \int_D v(x)^2 |x|^{-(2\delta+\gamma)} dx \right)^{1/2} \\
&\le b(u, u)^{1/2} b(v, v)^{1/2} \\
&\le \form(u, u)^{1/2} \form(v, v)^{1/2}.
\end{align*}
Thus, we obtain $|b(u, v)| \le (c+1)\form(u, u)^{1/2} \form(v, v)^{1/2}$. This ensures that the perturbed form 
$\form$  satisfies the sector condition and is well-defined as a semi-Dirichlet form.

Since $\delta < d-1$, the perturbed form $b(u,u)$ is nonnegative. This fact  immediately implies a fundamental energy ordering: $\form(u, u) = \form^0(u, u) + b(u, u) \ge \form^0(u, u)$. According to the definition of the class of measures of finite energy integral, any measure $\nu$ bounded by the energy of $\form^0$ is automatically bounded by the fortified energy of $\form$. Thus, the inclusion relation $\mathsf{S}_0(\form^0) \subseteq \mathsf{S}_0(\form)$ automatically holds for any $\delta \in (0, d-1)$.

We now investigate the structure of $\mathsf{S}_0(\form)$ in terms of the parameter $\delta$. The boundary is characterized by $\delta = 1-\gamma$, where the spatial singularity of the measure corresponds directly to the weight $|x|^{\gamma-2}$ appearing in the associated Hardy-type inequality.

\smallskip
\noindent
\underline{Case 1: $\delta \le 1-\gamma$ (Stable Phase).}

\smallskip
In this condition, the singularities of both the drift and the potential are controlled by the base energy $\form^0$ via the weighted Hardy inequality. Hence, the norm equivalence holds, and by Proposition 3.1, we have $\mathsf{S}_0(\form) = \mathsf{S}_0(\form^0)$.

\smallskip
\noindent
\underline{Case 2:  $\delta > 1-\gamma$ (Symmetry-Breaking Phase).}

\smallskip
By applying \(d\mu := |x|^{-(2\delta + \gamma)}dx\) and \(\beta_0=0\) to Proposition \ref{prop:break}, $\mathsf{S}_0(\form^0) \subsetneq \mathsf{S}_0(\form)$ holds. Indeed, for $\eta\in C_0^\infty([0,\infty))$ satisfying \( 0\leq \eta \leq 1\), \(\eta|_{|x|\leq 1} = 1\) and \(\eta|_{[2,\infty)}=0\), we set \( u_r(x):=\eta\!\left(|x|/r\right)\). By calculation, we have \( \|u_r\|_{L^2(E;\mu)}^2 \asymp r^{d-2\delta-\gamma}\) and \(\form^0(u_r,u_r) \asymp
r^{d+\gamma-2}.\)
Hence it holds that
\[\frac{\|u_r\|_{L^2(E;\mu)}^2}{\form^0(u_r,u_r)} \asymp r^{2-2\delta-2\gamma} \to \infty\]
as \(r \to 0\) and so \eqref{eq:break_1} holds.

We remark that, a standard radial integration analysis yields that \(\nu = |x|^{-\kappa} dx \) belongs to $\mathsf{S}_0(\form^0)$ if and only if $\kappa < \frac{d+2-\gamma}{2}$, and \(\nu\) belongs to  $\mathsf{S}_0(\form)$ if and only if $\kappa < \delta + \frac{d+\gamma}{2}$. Since $\delta > 1-\gamma$, we strictly have $\frac{d+2-\gamma}{2} < \delta + \frac{d+\gamma}{2}$. Consequently, it holds that \(\nu \in \mathsf{S}_0(\form) \setminus \mathsf{S}_0(\form^0)\)  for \(\kappa \in [\frac{d+2-\gamma}{2}, \delta + \frac{d+\gamma}{2})\).

This example shows that a stronger non-symmetric singularity leads to a larger class of measures of finite energy integral, which verifies the strict inclusion $\mathsf{S}_0(\form^0) \subsetneq \mathsf{S}_0(\form)$.

\end{exam}

\section{Non-Symmetric jump-type forms and forms associated with transposed jump kernels}\label{sec_jump}

In this section, we return to the general setting of $(E,{\sf d})$ and $\mathfrak{m}$ introduced in Section 1. Namely, $(E,{\sf d})$ is a locally compact separable metric space 
and $\mathfrak{m}$ is a positive Radon measure on $E$ with full support. 

Let $k(x,y)$ be a non-symmetric positive kernel defined on $(E \times E) \setminus \{(x,x) \mid x \in E\}$; that is, $k(x,y) \ge 0$ for $x \neq y$, 
and $k(x,y)$ is not identically equal to its transposed jump kernel $k^*(x,y)$, where $k^*(x,y) := k(y,x)$ (i.e., $k \not\equiv k^*$). We assume that $k$ satisfies
\begin{equation} \label{jump1}
\left\{
\begin{array}{l}
\dis x\mapsto \int_{y\not=x} \big(1 \wedge {\sf d} (x,y)^2\big) k_s(x,y) \mathfrak{m}(dy) \in L^1_{\sf loc}(E;\mathfrak{m}),  \\
\dis x\mapsto \int_{\substack{k_s(x,y)>0 \\ y \not=x}} \frac{|k(x,y)-k(y,x)|^2}{k_s(x,y)}\mathfrak{m}(dy) \in L^{\infty}(E;\mathfrak{m}),
\end{array}
\right.
\end{equation}
where $k_s(x,y) := (k(x,y) + k(y,x))/2$ denotes the symmetric part of $k$. 
For $u, v \in C_0^{\sf lip}(E)$, we define the bilinear form $\form$ by
$$
\form(u,v):= -\lim_{n\to \infty} \int_E {\mathscr L}^n u(x) v(x) \mathfrak{m}(dx), 
$$
whenever the limit on the right-hand side exists, where
$$
{\mathscr L}^nu(x):=\int_{{\mathsf d}(x,y)\ge 1/n} \big(u(y)-u(x) \big) k(x,y) \mathfrak{m}(dy), \quad x\in E
$$
for $n\in{\mathbb N}$. 
Under the conditions \eqref{jump1}, it follows from \cite{FOT11, O13} that $(\form, C_0^{\sf lip}(E))$ is closable on $L^2(E)$, and its closure $(\form, \dom)$ 
becomes a lower bounded semi-Dirichlet form with some lower bound parameter $\beta_0$ and the limit is expressed as follows for $u, v\in \dom$:
\begin{align*}
\form(u,v) & := \form^s(u,v) +\form^a(u,v) \\[10pt]
& := \frac 12 \iint_{x\not=y}\big(u(y)-u(x)\big) \big(v(y)-v(x)\big) k(x,y)\mathfrak{m}(dx)\mathfrak{m}(dy) \\
& \qquad +\frac12\iint_{x\not=y} \big(u(y)-u(x)\big) v(x) \big(k(x,y)-k(y,x)\big) \mathfrak{m}(dx)\mathfrak{m}(dy).
\end{align*}
Since $k^*(x,y) = k(y,x)$ also satisfies \eqref{jump1}, the bilinear form $\form^*$ associated with the transposed jump kernel $k^*$, is likewise a well-defined lower bounded semi-Dirichlet form on $L^2(E)$ with the same lower bound parameter $\beta_0$. By definition, the anti-symmetric part vanishes on the diagonal (i.e., $\form^a(u, u) = 0$) due to the anti-symmetry of the kernel $k(y,x)-k(x,y)$, ensuring the identity $\form_\beta(u,u) = \form^*_\beta(u,u)$ for all $u \in \dom$ and $\beta>\beta_0$. This quadratic equivalence, combined with the sector condition inherited from the kernel, ensures that the form $\form^*$ yields the exact same domain 
$D(\form^*) = \dom$ and determines the identical topology on the energy space. Thus, $\mathsf{S}_0(\form)=\mathsf{S}_0(\form^*)$ also holds.
Then, for $\nu \in\mathsf{S}_0(\form)$ and $\beta>\beta_0$, there exist $U_\beta \nu, \widehat{U}_\beta \nu, U^*_\beta \nu, \widehat{U}^*_\beta \nu \in \dom$  such that 
$$
\form_\beta(U_\beta \nu, v) = \form_\beta(v, \widehat{U}_\beta \nu)=\int_E v(x)\nu(dx)
$$
and 
$$
\form^*_\beta(U^*_\beta \nu, v) = \form^*_\beta(v, \widehat{U}^*_\beta \nu)=\int_E v(x)\nu(dx)
$$
hold for any $v\in \dom \cap C(E)$. 

However, the exact pointwise values of the associated potentials and co-potentials differ due to the spatial asymmetry of the jump kernel. 
The following proposition reveals that their exact discrepancies are intertwined and completely governed by the anti-symmetric part $\form^a$ acting as a continuous linear functional on the energy space, entirely bypassing any problematic singular integral representations (such as principal values) for the asymmetry.

\begin{prop}\label{prop:potential_relation}
For any $\nu \in \mathsf{S}_0(\form)$ and $\beta > \beta_0$, the potentials satisfy the following relations for any test function $v \in \dom$:
\begin{align}
\form^*_\beta(v, \widehat{U}^*_\beta \nu - U_\beta \nu) &= \form^a(U_\beta \nu, v) + \form^a(v, U_\beta \nu), \label{eq:relation_U_star} \\
\form_\beta(\widehat{U}^*_\beta \nu - U_\beta \nu, v) &= \form^a(v, \widehat{U}^*_\beta \nu) + \form^a(\widehat{U}^*_\beta \nu, v), \label{eq:relation_U} \\
\form^*_\beta(U^*_\beta \nu - \widehat{U}_\beta \nu, v) &= \form^a(v, \widehat{U}_\beta \nu) + \form^a(\widehat{U}_\beta \nu, v), \label{eq:relation_coU_star} \\
\form_\beta(v, U^*_\beta \nu - \widehat{U}_\beta \nu) &= \form^a(U^*_\beta \nu, v) + \form^a(v, U^*_\beta \nu). \label{eq:relation_coU}
\end{align}
\end{prop}

\begin{proof}
By the definitions of $\form^s$ and $\form^a$, we can decompose the forms as $\form_\beta = \form^s_\beta + \form^a$ and $\form^*_\beta = \form^s_\beta - \form^a$. 
First, we prove \eqref{eq:relation_U_star} and \eqref{eq:relation_U}. By the definition of the potentials $U_\beta \nu$ and $\widehat{U}^*_\beta \nu$, we have
\begin{align}
\form^s_\beta(U_\beta \nu, v) + \form^a(U_\beta \nu, v) &= \form_\beta(U_\beta \nu, v) = \int_E v(x)\nu(dx), \label{eq:def_U1} \\
\form^s_\beta(v, \widehat{U}^*_\beta \nu) - \form^a(v, \widehat{U}^*_\beta \nu) &= \form^*_\beta(v, \widehat{U}^*_\beta \nu) = \int_E v(x)\nu(dx), \label{eq:def_U2}
\end{align}
for any $v \in \dom \cap C(E)$. Since $\dom \cap C(E)$ is dense in $\dom$ with respect to the energy norm, these equations hold for all $v \in \dom$. Equating \eqref{eq:def_U1} and \eqref{eq:def_U2}, and using the symmetry $\form^s_\beta(U_\beta \nu, v) = \form^s_\beta(v, U_\beta \nu)$, we obtain
\begin{equation}\label{eq:sym_diff1}
\form^s_\beta(v, \widehat{U}^*_\beta \nu - U_\beta \nu) = \form^a(U_\beta \nu, v) + \form^a(v, \widehat{U}^*_\beta \nu).
\end{equation}
We then evaluate the left-hand side of \eqref{eq:relation_U_star}. By the decomposition of $\form^*_\beta$ and substituting \eqref{eq:sym_diff1}, we have
\begin{align*}
\form^*_\beta(v, \widehat{U}^*_\beta \nu - U_\beta \nu) 
&= \form^s_\beta(v, \widehat{U}^*_\beta \nu - U_\beta \nu) - \form^a(v, \widehat{U}^*_\beta \nu - U_\beta \nu) \\
&= \Big( \form^a(U_\beta \nu, v) + \form^a(v, \widehat{U}^*_\beta \nu) \Big) - \form^a(v, \widehat{U}^*_\beta \nu) + \form^a(v, U_\beta \nu) \\
&= \form^a(U_\beta \nu, v) + \form^a(v, U_\beta \nu),
\end{align*}
which shows \eqref{eq:relation_U_star}. Similarly, evaluating the left-hand side of \eqref{eq:relation_U} yields
\begin{align*}
\form_\beta(\widehat{U}^*_\beta \nu - U_\beta \nu, v) 
&= \form^s_\beta(\widehat{U}^*_\beta \nu - U_\beta \nu, v) + \form^a(\widehat{U}^*_\beta \nu - U_\beta \nu, v) \\
&= \Big( \form^a(U_\beta \nu, v) + \form^a(v, \widehat{U}^*_\beta \nu) \Big) + \form^a(\widehat{U}^*_\beta \nu, v) - \form^a(U_\beta \nu, v) \\
&= \form^a(v, \widehat{U}^*_\beta \nu) + \form^a(\widehat{U}^*_\beta \nu, v),
\end{align*}
which shows \eqref{eq:relation_U}.

We can obtain \eqref{eq:relation_coU_star} and \eqref{eq:relation_coU} similarly.
\end{proof}

Furthermore, if the asymmetric jump kernel possesses sufficient integrability, the relationships in Proposition \ref{prop:potential_relation} 
can be simplified. Suppose that the mapping $y \mapsto k(x, y) - k(y, x)$ is integrable with respect to $\mathfrak{m}$, and the function
\begin{equation}
K(x) := \int_{y \neq x} \big(k(x, y) - k(y, x)\big) \mathfrak{m}(dy)
\end{equation}
is well-defined and belongs to $L^1_{\sf loc}(E)$. In this case, the sum of the anti-symmetric parts $\form^a(u, v) + \form^a(v, u)$ reduces to a simple multiplication operator. 
Indeed, by swapping the variables $x$ and $y$, the cross terms $u(y)v(x)$ and $u(x)v(y)$ are canceled out due to the anti-symmetry of $k(y, x) - k(x, y)$, leaving only the diagonal term:
\begin{align*}
\form^a(u, v) + \form^a(v, u) &= \frac{1}{2} \iint\limits_{x \neq y} \big(u(y)v(x) + u(x)v(y) - 2u(x)v(x)\big) \big(k(y, x) - k(x, y)\big) \mathfrak{m}(dx)\mathfrak{m}(dy) \\
&= - \iint\limits_{x \neq y} u(x)v(x) \big(k(y, x) - k(x, y)\big) \mathfrak{m}(dx)\mathfrak{m}(dy) \\
&= \int\limits_E u(x)v(x) K(x) \mathfrak{m}(dx).
\end{align*}
This $K(x)$ probabilistically represents the rate of killing (or creation) of mass. Because of this structural discrepancy caused by the nonsymmetric
jumps, neither the dual form $\widehat{\form}$ nor the dual form of $\form^*$ satisfies the Markov property. Thus, the time-reversed dual processes inherently 
experience spatial killing, making them semi-Dirichlet forms.

Substituting this identity into Proposition \ref{prop:potential_relation} yields the following explicit perturbation formula.

\begin{cor}  \label{cor:section6}

Let $\nu \in \mathsf{S}_0(\form)$ and $\beta > \beta_0$. Assume that $K \in L^1_{\sf loc}(E)$ as above. If the measures $U_\beta \nu \cdot K \mathfrak{m}$ and $\widehat{U}_\beta \nu \cdot K \mathfrak{m}$ belong to $\mathsf{S}_0(\form^*) = \mathsf{S}_0(\form)$, then the potentials satisfy the following :
$$
\widehat{U}^*_\beta \nu = U_\beta \nu + \widehat{U}^*_\beta(U_\beta \nu \cdot K \mathfrak{m}) = U_\beta \nu + U_\beta(\widehat{U}^*_\beta \nu \cdot K \mathfrak{m}). 
$$
Similarly, we have
$$
U^*_\beta \nu = \widehat{U}_\beta \nu + U^*_\beta(\widehat{U}_\beta \nu \cdot K \mathfrak{m})= \widehat{U}_\beta \nu + \widehat{U}_\beta(U^*_\beta \nu \cdot K \mathfrak{m}). 
$$
\end{cor}

\smallskip
It should be noted that the identities in Corollary \ref{cor:section6} are not explicit formulas, but rather represent \textit{resolvent-type integral equations}. 
They explicitly describe how the potentials are intertwined through the asymmetric perturbation $K\mathfrak{m}$.

\begin{exam} \rm 
Let us consider a stable-like kernel on 
$E = \real^d$ ($d \ge 3$). For $0 < \alpha < 2$ and $\gamma>0$, we define
\begin{equation*}
k(x, y) := a(x)|x-y|^{-d-\alpha}, \quad \text{where } a(x) := (1+|x|^2)^{\gamma/2}.
\end{equation*}
Here, the variable jump rate $a(x)$ smoothly diverges at spatial infinity. If we restrict the growth rate to be strictly slower than the jump index, namely $0 < \gamma < \alpha$, the condition (4.2) is satisfied. Indeed, the anti-symmetric component behaves as 
$$
\frac{|a(x)-a(y)|^2}{a(x)+a(y)}|x-y|^{-d-\alpha} \sim |y|^{\gamma-d-\alpha} \quad {\rm as} \quad |y| \to \infty,
$$
which is integrable at infinity since $\gamma < \alpha$. Thus, the kernel yields a well-defined lower bounded semi-Dirichlet form.

Let 
$$
\mathscr{L}_\alpha u(x) := {\rm P.V.} \int_{\mathbb{R}^d} (u(y)-u(x))|x-y|^{-d-\alpha}dy
$$ 
be the generator of a symmetric $\alpha$-stable process.  In our kernel $k(x,y)$, the non-local generator associated with $(\form, \dom)$ acts as 
${\mathscr L}u(x) = a(x)\mathscr{L}_\alpha u(x)$. 

For a Radon measure $\nu(dx) = f(x)dx$ with  a nonnegative density $f$ belonging to $C_0(E)$,  the $\beta$-potential $u = U_\beta \nu$ (resp. the co-potential $w = \widehat{U}_\beta \nu$)  is simply the resolvent (resp. co-resolvent) of $f$,  and these potentials formally solve the equations:
\begin{equation}\label{eq:governing-L}
-\mathscr{L} u + \beta u = f \quad \text{and} \quad -\mathscr{L}^* w + \beta w = f,
\end{equation}
where $\mathscr{L}^*$ denotes the adjoint operator of ${\mathscr L}$ on $L^2(\mathbb{R}^d)$. Since ${\mathscr L}=a{\mathscr L}_\alpha$ and 
the symmetric operator $\mathscr{L}_\alpha$ is self-adjoint on $L^2(\mathbb{R}^d)$, the self-adjointness of $\mathscr{L}_\alpha$ immediately implies that the adjoint operator acts as $\mathscr{L}^* w = \mathscr{L}_\alpha (aw)$. Thus, we can rewrite the equations in \eqref{eq:governing-L} explicitly in terms of $\mathscr{L}_\alpha$ as follows:
\begin{equation}\label{eq:governing-alpha}
-a(x)\mathscr{L}_\alpha u(x) + \beta u(x) = f(x) \quad \text{and} \quad -\mathscr{L}_\alpha (aw)(x) + \beta w(x) = f(x).
\end{equation}
By introducing $V(x) := a(x)w(x)$, the adjoint equation is rewritten as 
$$
-\mathscr{L}_\alpha V + \frac{\beta}{a(x)}V = f(x).
$$
As $|x| \to \infty$, the variable absorption term $\frac{\beta}{a(x)}$ decays to zero since $a(x) \sim |x|^\gamma$. Because $f$ has compact support and $V$ is a bounded solution, the term $\frac{\beta}{a(x)}V(x)$ vanishes at infinity and acts as a localized perturbation. Consequently, the leading asymptotic behavior of both $u$ (from the first equation in \eqref{eq:governing-alpha} after dividing by $a(x)$) and $V$ is effectively governed by the Riesz kernel of 
$-\mathscr{L}_\alpha$, ensuring that $V(x) = \mathscr{O}(|x|^{-d+\alpha})$. 
We thus obtain the sharp spatial decay rates at infinity:
\begin{equation*}
u(x) = \mathscr{O}(|x|^{-d+\alpha}) \quad \text{and} \quad w(x) = \frac{V(x)}{a(x)} = \mathscr{O}(|x|^{-d+\alpha-\gamma}) \quad \text{as } |x| \to \infty.
\end{equation*}
By applying this to $f$ that approximates the equilibrium measure of a compact set $A$, this result translates to the spatial decay of the equilibrium potentials. 

{
This difference between the pointwise asymmetry and the global energy identity can be explicitly observed through the notion of capacity. 
For a relatively compact open set $A$, the capacity is given by the mutual energy $\text{Cap}^{(\alpha)}(A) = \form_\alpha(e^\alpha_A, \widehat{e}^\alpha_A)$, where $e^\alpha_A$ and $\widehat{e}^\alpha_A$ are the $\alpha$-equilibrium and co-equilibrium potentials. 
Probabilistically, these represent the hitting probabilities to the set $A$ from the outside for the forward process and the time-reversed process, respectively. 
As demonstrated here, these two potentials exhibit distinct spatial decay rates at spatial infinity, namely, 
$e^\alpha_A(x)={\mathscr O}(|x|^{-d+\alpha})$ and $\widehat{e}^\alpha_A(x) ={\mathscr O}(|x|^{-d+\alpha-\gamma})$, while their mutual energy still yields the finite capacity. 
}

This phenomenon explicitly visualizes how the spatial non-symmetry of the jump rate generates the intrinsic killing effect in the dual process.

\end{exam}

\begin{rem} \rm
\begin{itemize}
\item[\sf (1)]   
We briefly note the local regularity of the potentials here by invoking the Schauder-type regularity results  for stable-like operators developed by Bass 
\cite[Theorem 1.2, Proposition 7.4]{B09}.  Suppose that the density $f$ is locally $\beta$-H\"older continuous for some $\beta > 0$. Since the coefficient $a(x) = (1+|x|^2)^{\gamma/2}$ is smooth and locally bounded away from zero, Bass's result ensures that the potential $u = U_\beta \nu$ is locally $C^{\alpha+\beta}$. Similarly, for the co-potential $w = \widehat{U}_\beta \nu$, the transformed function $V = aw$ satisfies an equation principally governed by the classical symmetric fractional Laplacian ${\mathscr L}_\alpha$. 
This implies that $V$, and consequently $w = V/a$, also possesses the same local $C^{\alpha+\beta}$ regularity.

This observation reveals that the spatial asymmetry of the jump rate $a(x)$ does not disrupt the local H\"older continuity. Consequently, the discrepancy between the potential $u=U_\beta \nu$ and the co-potential $w=\widehat{U}_\beta \nu$ appears solely as a structural difference in their global behaviors at spatial infinity, as seen in their decay rates above.

\item[\sf (2)]  While the specific choice of the kernel $k(x,y)=a(x)|x-y|^{-d-\alpha}$  in this example might be probabilistically related 
to a time-changed symmetric $\alpha$-stable process, it serves as the most transparent and analytically tractable model to explicitly observe the discrepancy of spatial decay rates.  A more general non-symmetric kernel, such as 
$(a(x)+b(y))|x-y|^{-d-\alpha}$ with different growth orders for $a$ and $b$, would essentially exhibit a similar discrepancy phenomenon. However, finding their exact asymptotic behavior requires highly involved singular integral estimates, which is beyond the scope of this paper.

\end{itemize}
\end{rem}


\end{document}